\documentclass[12pt,a4paper,reqno]{amsart}
\DeclareMathSymbol{\twoheadrightarrow}
{\mathrel}{AMSa}{"10}

\def\Q{{\mathbb Q}}
\def\Z{{\mathbb Z}}
\def\C{{\mathbb C}}
\def\R{{\mathbb R}}
\def\F{{\mathbb F}}
\def\R{{\mathfrak R}}

\def\bphi{{\mathbf \phi}}

                                \def\one{{\mathbf 1}}

\def\Sn{{\mathbf S}_n}
\def\An{{\mathbf A}_n}

\def\Gal{\mathrm{Gal}}
\def\Perm{\mathrm{Perm}}
\def\Alt{\mathrm{Alt}}

\def\End{\mathrm{End}}
\def\Aut{\mathrm{Aut}}

\def\Mat{\mathrm{Mat}}

\def\Id{\mathrm{Id}}

\def\fchar{\mathrm{char}}

\def\GL{\mathrm{GL}}
\def\SL{\mathrm{SL}}

\def\PGL{\mathrm{PGL}}

\def\dim{\mathrm{dim}}

\def\delt{{\boldsymbol \delta}}

\newtheorem{thm}{Theorem}[section]
\newtheorem{lem}[thm]{Lemma}
\newtheorem{cor}[thm]{Corollary}
\newtheorem{prop}[thm]{Proposition}
\theoremstyle{definition}
\newtheorem{defn}[thm]{Definition}

\newtheorem{rem}[thm]{Remark}
\newtheorem{rems}[thm]{Remarks}
\title[Superelliptic  Jacobians]
{Superelliptic Jacobians and central simple representations}
\author[Yuri G. Zarhin]{Yuri G. Zarhin}
\address{Department of Mathematics, Pennsylvania State University,
University Park, PA 16802, USA}
\email{zarhin\char`\@math.psu.edu}
\thanks{Partially supported by the Simons Foundation Collaboration grant \# 585711.
Part  of this work was done in  December 2023 during my stay at the Max-Planck Institut f\"ur Mathematik (Bonn, Germany), whose hospitality and support are gratefully acknowledged.}

\begin{document}
\begin{abstract}
Let $f(x)$ be a polynomial of degree at least $5$ with complex coefficients and without repeated roots. Suppose that all the coefficients of $f(x)$ lie in a subfield $K$ of $\C$ such that: 
\begin{itemize}
\item $K$ contains a primitive p-th root of unity;
\item
 $f(x)$ is irreducible over $K$;
\item
  the Galois group $\Gal(f)$ of $f(x)$ acts doubly transitively on the set of roots of $f(x)$;
\item
 the index of every maximal subgroup of $\Gal(f)$ does {\sl not} divide $\deg(f)-1$. 
 \end{itemize}
Then the endomorphism ring of the Jacobian of the superelliptic curve $y^p=f(x)$ is isomorphic to the $p$th cyclotomic ring for all primes $p>\deg(f)$. 
\end{abstract}
\maketitle
\section{Introduction}
\label{Introd}
The aim of this paper is to explain how to compute the endomorphism algebra of Jacobians of smooth projective models of superelliptic curves $y^q=f(x)$ where $q=p^r$ is a prime power and $f(x)$ a polynomial of degree $n \ge 5$ with complex coefficients that is in ``general position''. Here
``general position'' means that there is a (sub)field $K$ such that all the coefficients of $f(x)$ lie in $K$ and the Galois group of $f(x)$ acts doubly transitively on the set of its roots (in particular, $f(x)$ is irreducible over $K$).  It turns out that for a broad class of the doubly transitive Galois groups (and under certain mild restrictions on $q$) the corresponding endomorphism algebra is ``as small as possible'', i.e., is canonically isomorphic to a product of cyclotomic fields $\Q(\zeta_{p^i})$  ($1 \le i \le r$).

 In order to state explicitly our results, let us start with the notation and some basic facts related to cyclotomic fields and cyclotomic polynomials.
As usual, $\Z,\Q,\C$ denote the ring of integers, the field of rational numbers and the field of complex numbers respectively.

Let $p$ be an odd prime and $\F_p$  the corresponding (finite) prime field of characteristic $p$. We write $\Z_p$ and $\Q_p$
for the ring of $p$-adic integers and the field $\Q_p$ of $p$-adic numbers respectively.
Let $r$ be a positive integer and $q=p^r$.
Let 
$$\zeta_q \in \C.$$
be a primitive $q$th root of unity. We write $\Q(\zeta_q)$ be the $q$th cyclotomic field and 
$$\Z[\zeta_q]=\sum_{i=0}^{\phi(q)-1}\Z\cdot \zeta_q^i$$
for its ring of integers. (Hereafter $\phi(q):=(p-1)p^{r-1}$ is the Euler function.)

Let us consider the polynomial
$$\mathcal{P}_q(t):=\sum_{j=0}^{q-1} t^j =\prod_{i=1}^r \Phi_{p^i}(t)\in \Z[t]$$
where 
$$\Phi_{p^i}(t)=\sum_{j=0}^{p-1} t^{i p^{r-1}}\in \Z[t]$$
is the $p^i$th {\sl cyclotomic} polynomial.

Let $f(x) \in \C[x]$ be a polynomial of degree $n\ge 4$ without repeated roots.  In what follows we always assume that {\sl either  $p$ does not divide $n$ or $q$ divides $n$}.

Let $C_{f,q}$ be a smooth projective model of the smooth affine curve
$$y^q=f(x).$$
It is well known
(\cite{Koo}, pp. 401--402, \cite{Towse}, Prop. 1 on p. 3359,
 \cite{Poonen}, p. 148)
that the genus $g(C_{f,p})$  of $C_{f,p}$ is $(q-1)(n-1)/2$ if $p$ does not divide $n$ and $(q-1)(n-2)/2$ if $q$ divides $n$.
The map
 $$(x,y) \mapsto (x, \zeta_p y)$$
gives rise to a non-trivial biregular automorphism
$$\delta_q: C_{f,q} \to C_{f,q}$$
of period $q$.

Let $J(C_{f,q})$  be the Jacobian of $C_{f,q}$;
it is a $g(C_{f,q})$-dimensional abelian variety.
We write $\End(J(C_{f,q}))$ for the ring of endomorphisms of $J(C_{f,q})$ and
$\End^0(J(C_{f,q}))=\End(J(C_{f,q}))\otimes\Q$ for the endomorphism algebra of $J(C_{f,q})$.
By functoriality,
$\delta_q$ induces an automorphism of $J(C_{f,q})$,
which we still denote by $\delta_q$.  It is known (\cite[p. 149]{Poonen}, \cite[p. 448]{SPoonen}, 
\cite[Lemma 4.8]{ZarhinM}) that
\begin{equation}
\label{deltaEquation}
 \mathcal{P}_q(\delta_q)=0
\end{equation}
in $\End(J(C_{f,q}))$. 
Then \eqref{deltaEquation} gives rise to the  ring homomorphism,
\begin{equation}
\label{embedQ}
{\bf i}_{q,f}: \Z[t]/\mathcal{P}_q(t)\Z[t] \hookrightarrow \Z[\delta_q] \subset \End(J(C_{f,q})), \ t+\mathcal{P}_q(t)\Z[t]  \mapsto \delta_q,
\end{equation}
which is a {\sl ring embedding} (\cite[p. 149]{Poonen}, \cite[p. 448]{SPoonen}, 
\cite[Lemma 4.8]{ZarhinM}). (The first map in \eqref{embedQ} is actually a ring isomorphism.)
 This implies  that the subring $\Z[\delta_q]$ of $\End(J(C_{f,q}))$ generated by $\delta_q$
is isomorphic to $\Z[t]/\mathcal{P}_q(t)\Z[t]$. It follows that  the  $\Q$-subalgebra 
\begin{equation}
\label{Qdeltaq}
\Q[\delta_q] \subset\End^0(J(C_{f,q}))
\end{equation}
 generated by $\delta_q$ has $\Q$-dimension $q-1$, 
is isomorphic to 
$$\Q[t]/\mathcal{P}_q(t)\Q[t]\cong \prod_{i=1}^r \Q(\zeta_{p^i})$$
and therefore has dimension $q-1$.

We will need the following elementary observation.

\begin{rem}
\label{pNdiv}
\begin{itemize}
\item[(i)]
Suppose that a prime $p$ is greater than $n$. Then $p$ does {\sl not} divide $n!$. Since every subgroup $H$ of $\Gal(f)$ is isomorphic to a subgroup of $\mathbf{S}_n$,   its order $|H|$ divides $n!$ and therefore is {\sl not} divisible by $p$. Hence, if $p>n$, then  $|H|$  is {\sl not} divisible by $p$.
\item[(ii)]
Suppose that $H$ is a transitive subgroup of $\Gal(f)$ with respect to the action on the roots of $f(x)$.  Then its order $|H|$ is divisible by $n$.
\end{itemize}
\end{rem}

Let us formulate our main results.

First, we start with the case $q=p$. Then $\mathcal{P}_p(t)$ coincides with  $\Phi_p(t)$ and there is a natural ring isomorphism
$$\Z[t]/\mathcal{P}_p(t)\Z[t] \cong \Z[\zeta_p]$$
that sends (the coset of) $t$ to $\zeta_p$.
This gives us the the {\sl ring embedding}
\begin{equation}
\label{embedP}
{\bf i}_{p,f}: \Z[\zeta_p] \hookrightarrow \Z[\delta_p] \subset \End(J(C_{f,p})), \ \zeta_p  \mapsto \delta_p.
\end{equation}
Notice also that the rings $\Z[\delta_p]$ and $\Z[\zeta_p]$ are isomorphic.
\begin{thm}
\label{primeToP}
Let $n \ge 5$ be an integer and $p$ an odd prime such that   $K$ contains a primitive $p$th root of unity.

Suppose that the Galois group $\Gal(f)$ of $f(x)$ contains a subgroup $H$  that acts doubly transitively
on the  $n$-element set $\R_f$ of roots of the polynomial $f(x)$ and enjoys the following properties.
\begin{itemize}
\item[(i)]
 The index of every maximal subgroup of  $H$  does not divide $n-1$.
\item[(ii)]
 $p$   does not divide  $|H|$. (E.g., $p>n$.)
\end{itemize}
Then $\End^0(J(C_{f,p}))=\Q[\delta_p]$ and $\End(J^{(f,p)})=\Z[\delta_p]$.
\end{thm}

\begin{thm}
\label{endo}
Let $K$ be a subfield of $\C$ such that all the coefficients of $f(x)$
lie in $K$. Assume also that $f(x)$ is an irreducible polynomial in $K[x]$
of degree $n \ge 5$
 and its Galois group over $K$ is either the full symmetric group $\Sn$ or
the alternating group $\An$.
Then
$$ \End(J(C_{f,p}))=\Z[\delta_p] \cong\Z[\zeta_p].$$
\end{thm}

\begin{rem}
\label{errorA5}
Theorem \ref{endo} was stated in \cite[Th. 4.2]{ZarhinCambridge}. Its proof was based on an assertion that  
a certain  ``permutational'' representation $(\F_p^B)^{00}$ (that is called the {\sl heart}
\footnote{See \cite{Mortimer,Klemm} and Section \ref{permute} below for the definition of the {\sl heart}.})
of the alternating group $\Alt(B)=\An$ over $\F_p$ is  {\sl very simple}
 \footnote{See \cite{ZarhinM,ZarhinClifford} and Section \ref{verySR} below for the definition and basic properties of very simple representations.}
  \cite[Th. 4.7]{ZarhinCrelle}. Unfortunately, 
there is an error in the proof of \cite[Th. 4.7]{ZarhinCrelle}  when $n=5, p>5$,  caused by an improper use of \cite[Cor. 4.4]{ZarhinCrelle} (see \cite[p. 108, lines 4-5]{ZarhinCrelle}).
So, the proof in \cite{ZarhinCambridge} works only under an additional assumption that either $n>5$ or $p \le 5$.

In this note we handle the remaining case when $n=5, \ p>5$. It turns out that if $p \not \equiv \pm 1 \ \bmod 5$
then the representation of the group $\mathbf{A}_5$ is very simple, which  allows us to salvage in this case the  arguments of \cite{ZarhinCrelle}.

However,  if $p\equiv \pm 1 \ \bmod 5$ then  the $4$-dimensional representation
$\left(\F_p^B\right)^{00}$ of  $\mathbf{A}_5$, viewed as the representaton of the covering group $\mathrm{SL}(2,\F_5)$, splits into a tensor
product of two $2$-dimensional representations. In particular, $\left(\F_p^B\right)^{00}$ is {\sl not} very simple, and we use a notion of a {\sl central simple representation}
(see Section \ref{verySR} below), in order 
to prove Theorem \ref{endo} in this case.
\end{rem}

\begin{rems}
\label{ribcom}

If $f(x) \in K[x]$ then the curve
$C_{f,p}$ and its Jacobian $J(C_{f,p})$ are defined over $K$. Let $K_a\subset \C$ be the algebraic closure of $K$. 
Then all endomorphisms of $J(C_{f,p})$ are defined over $K_a$. Hence,  
 in order to prove Theorems \ref{endo} and \ref{SimpleDT}, it suffices to
check that the ring of all $K_a$-endomorphisms of  $J(C_{f,p})$ coincides with
 $\Z[\delta_p]$. 
 \end{rems}
 
 Now let us try to relax the restrictions on $p$, keeping the double transitivity of $\Gal(f)$. 
  Our next result deals with doubly transitive sporadic simple  (Galois) groups,
 whose description may be found in \cite{PraegerS}, \cite[Ch. 6 and Ch. 7, Sect. 7.7, p. 252-253]{DM}.

 \begin{thm}
 \label{mathieu}
 Let $p$ be an odd prime and $$\Gal(f)\subset \mathrm{Perm}(\R_f) \cong \mathbf{S}_n$$  a  permutation group that  acts doubly transitively on  the $n$-element set $\R_f$.
 Suppose that $(n,\Gal(f))$  enjoys one of the following properties.
 \begin{itemize}
 \item[({\bf M})]
 $n \in \{11,12,22,23,24\}$,  and $\Gal(f)$
 is isomorphic to the corresponding Mathieu group $\mathbf{M}_n$. 
 If $n=11$ then we assume additionally that $p>3$.
 \item[({\bf HS})]
 $n=176, \ p>7$, and $\Gal(f))$ is isomorphic to the sporadic simple Higman-Sims group $\mathrm{HS}$.
 \item[({\bf CO3})]
 $n=276$,   $p \not\in  \{3,5,11\}$, and $\Gal(f))$  is isomorphic to the sporadic simple Conway group $\mathrm{Co}_3$.
 \end{itemize}
 Then $\End^0(J(C_{f,p}))=\Q[\delta_p]$ and $\End(J(C_{f,p}))=\Z[\delta_p]$.
\end{thm}

\begin{rem}
The case ({\bf M}) of Theorem \ref{mathieu} gives a (partial) answer to a question of Ravi Vakil that was asked during my talk at 
Simons  Symposium ``Geometry Over Non-closed Fields'' (Puerto Rico, March 2015).
 \end{rem}

Now let us discuss the case when our Galois groups are doubly transitive finite simple {\sl Chevalley groups} - they are classified in \cite{Curtis} and their action described in details
in \cite[Sect. 7.7]{DM}. (For general results about Chevalley groups see \cite{Steinberg}.) 
  
   \begin{thm}
 \label{PSL2}
 Suppose that $p$ be an odd prime and $\Gal(f)$ contains a subgroup $H$ that  acts doubly transitively on  the $n$-element set $\R_f$ and  is isomorphic to a finite Chevalley group $\mathfrak{G}(\mathfrak{q})$, and the corresponding stabilizers correspond to  Borel subgroups of $\mathfrak{G}(\mathfrak{q})$, which are maximal subgroups of index $n$.
 
 Suppose that $(n,\mathfrak{G}(\mathfrak{q}),p)$  enjoys one of the following properties.
 
 \begin{itemize}
 \item[({\bf L2})]
  Let $\ell$ be a prime and $\mathfrak{r}$ a positive integer.
 Then $n=\mathfrak{q}+1$ where $\mathfrak{q}=\ell^{\mathfrak{r}}>11$ and
 $\mathfrak{G}(\mathfrak{q})$ is the projective special linear group $$\mathbf{L}_2(\mathfrak{q})=\mathrm{PSL}(2,\F_{\mathfrak{q}})$$
  where $\F_{\mathfrak{q}}$ is a finite $\mathfrak{q}$-element field.
Assume additionally that either $p \ne \ell$ or $\mathfrak{q}=\ell=p$.
 \item[({\bf Lmq})]
  Let $m \ge 3$ be an integer, $\ell$  a prime, $\mathfrak{r}$ a positive integer.
 Then $n=(\mathfrak{q}^m-1)/(\mathfrak{q}-1)$ where $\mathfrak{q}=\ell^{\mathfrak{r}}$ and
 $\mathfrak{G}(\mathfrak{q})$ is the projective special linear group $$\mathbf{L}_m(\mathfrak{q})=\mathrm{PSL}(m,\F_{\mathfrak{q}})$$
  where $\F_{\mathfrak{q}}$ is a finite $\mathfrak{q}$-element field.
Assume additionally that  $p \ne \ell$ and
$$(m, \mathfrak{q}) \ne (3,2), (3,4), (4,2), (4,3), (6,2), (6,3).$$

 \item[({\bf U3})]
 Let $\ell$ be a prime and $\mathfrak{r}$ a positive integer.
 Then $n=\mathfrak{q}^3+1$ where $\mathfrak{q}=\ell^{\mathfrak{r}}$ is a power  of a prime $\ell$,
 $$\mathfrak{q}\ne 2,5$$
and $\mathfrak{G}(\mathfrak{q})$ is  the projective special unitary group $\mathbf{U}_3(q)=\mathrm{PSU}_3(\F_{q})$.
\item[({\bf Sz})] Let $\mathfrak{r}$ be a positive integer, 
$$\mathfrak{q}=2^{2\mathfrak{r}+1}, \quad n=q^2+1, \ m =2^{\mathfrak{r}+1}.$$ 
 Then $H$ is the simple Suzuki group $\mathrm{Sz}(\mathfrak{q})=^2\mathrm{B}_2(q)$. In addition,  $p$ does not divide $(q+1+m)$.
 \item[({\bf Ree})]
 Let $\mathfrak{r}$ be a positive integer, 
$$\mathfrak{q}=3^{2\mathfrak{r}+1}, \quad n=q^2+1, \ m =3^{\mathfrak{r}+1}.$$ 
The group $H$ is the simple Ree group $\mathrm{Ree}(\mathfrak{q})=^2\mathrm{G}_2(q)$. In addition, 
$p$ does not divide $3(q+1)(q+m+1)(q-m+1)$.
 \end{itemize}
 Then $\End^0(J(C_{f,p}))=\Q[\delta_p]$ and $\End(J(C_{f,p}))=\Z[\delta_p]$.
 
\end{thm}

Now let us assume that $r$ is any positive integer (recall that $q=p^r$). In this case we obtain the results about the {\sl endomorphism algebra} $\End^0(J(C_{f,q}))=\End(J(C_{f,q}))\otimes \Q$
of $J(C_{f,q})$
that may be viewed as analogues of Theorems \ref{endo} and Theorem \ref{primeToP} for the endomorphism algebra  $\End^0(J(C_{f,q})$.

\begin{thm}
\label{primeToQ}
Suppose that $n \ge 5$ is an integer, $p$ an odd prime, $q$ divides $n$, and  $K$ contains a primitive $q$th root of unity.

Let us assume that the Galois group $\Gal(f)$ of $f(x)$ contains a subgroup $H$  that acts doubly transitively
on the  $n$-element set $\R_f$ of roots of the polynomial $f(x)$ and enjoys the following properties.
\begin{itemize}
\item[(i)]
 The index of every maximal subgroup of  $H$  does not divide $n-1$.
\item[(ii)]
 $p$   does not divide  $|H|$. (E.g., $p>n$.)
\end{itemize}

Then 
$$\End^0(J(C_{f,q}))=\Q[\delta_q]\cong \prod_{i=1}^r \Q(\zeta_{p^i}).$$
 \end{thm}

\begin{thm}
\label{endoQ}
Let $K$ be a subfield of $\C$ such that all the coefficients of $f(x)$
lie in $K$. Suppose that $f(x)$ is an irreducible polynomial in $K[x]$
of degree $n \ge 5$
 and its Galois group over $K$ is either the full symmetric group $\Sn$ or
the alternating group $\An$. Assume also that either $p$ does not divide $n$ or $q$ divides $n$.
Then
$$\End^0(J(C_{f,q}))=\Q[\delta_q] \cong \prod_{i=1}^r \Q(\zeta_{p^i}).$$

\end{thm}

\begin{rem}
\label{errorA5bis}
Theorem \ref{endoQ} was stated in \cite[Th. 1.1]{ZarhinM}. 
Similarly (see Remark \ref{errorA5}), 
the proof in \cite{ZarhinM} works only under an additional assumption that either $n>5$ or $p \le 5$. In this paper we handle the remaining case $n=5, p>5$.
\end{rem}

\begin{rem}
An analogue of  Theorem \ref{endoQ} when $p \mid n$ but $q$ does {\sl not} divide $n$ was proven in \cite{Jiangwei}.

\end{rem}

The paper is organized as follows. In Section \ref{weightS} we discuss complex abelian varieties  $Z$ with multiplications from cyclotomic fields, paying special attention 
to   the centralizers of these fields in $\End^0(Z)$ and their action on the differentials of the first kind when $Z$ is a superelliptic Jacobian In Section \ref{permute}
we discuss modular representations of permutation groups, paying a special attention to the {\sl hearts} of these representations.
In Section \ref{verySR} we introduce central simple representations and recall basic properties of very simple representations that first appeared in \cite{ZarhinM,ZarhinClifford},
paying a special attention to the very simplicity and central simplicity of hearts of permutational representations in the case of doubly transitive permutation groups.
In Section \ref{cyclotomicA} we return to our discussion of abelian varieties  $Z$ with multiplications from cyclotomic fields $\Q(\zeta_q)$, paying a special attention
to the Galois properties of the set of $\delta_q$-invariants. In Section \ref{Prelim} we review results of \cite{ZarhinCambridge} about endomorphism algebras
of superelliptic Jacobians. In Section \ref{mainPproof} we prove our main results that deal with the case $q=p$. The proofs for the case of arbitrary $q$ are contained in 
Section \ref{mainQproof}.

{\bf Acknowledgements}. I am grateful to the referee, whose comments helped to improve the exposition.

\section{Endomorphism fields of abelian varieties and their centralizers}
\label{weightS}
In what follows $E$ is a number field and $O_E$ the ring of algebraic integers in $E$. It is well known that $O_E$
is a Dedekind ring and therefore every finitely generated torsion-free $O_E$-module is projective/locally free and isomorphic 
to a direct sum of locally free $O_E$-modules of rank $1$. In addition, the natural map
$$O_E\otimes \Q \to E, \ e\otimes c \mapsto c\cdot e \ \forall e\in E, \ c \in \Q$$
is an isomorphism of $\Q$-algebras.

Let $Z$ be an abelian variety over $\C$ of positive dimension $g$, let  $\End(Z)$ be the ring of its endomorphisms.
  If $n$ is an integer then we we write $n_Z$ for the endomorphism
$$n_Z: Z \to Z, \ z \mapsto nz.$$
Clearly, $n_Z \in \End(Z)$. By definition, $1_Z$ is the identity selfmap of $Z$. In addition,
$n_Z: Z \to Z$ is an {\sl isogeny} if and only if $n \ne 0$.

We write 
$$\End^0(Z)=\End(Z)\otimes\Q$$
for the corresponding endomorphism algebra of $Z$, which is a finite-dimensional semisimple $\Q$-algebra. Identifying $=\End(Z)$ with 
$$\End(Z)\otimes 1\subset \End(Z)\otimes\Q=\End^0(Z),$$
we may view $\End(Z)$ as an {\bf order} in the $\Q$-algebra $\End^0(Z)$.

The action of $\End(Z)$  by functoriality  on the $g$-dimensional complex vector space $\Omega^1(Z)$ of differentials of the first kind on $Z$ 
 gives us the ring homomorphism  $\End(Z) \to  \End_{\C}(\Omega^1(Z)$ \cite[Ch. 1, Sect. 2.8]{Shimura},
which extends by $\Q$-linearity to the homomorphism of $\Q$-algebras
\begin{equation}
\label{dif1kind}
j_Z:\End^0(Z)\to \End_{\C}(\Omega^1(Z)),
\end{equation}
which sends $1_Z$ to the identity automorphism of the $\C$-vector space $\Omega^1(Z)$.
Let $E$ be a  number field that is a $\Q$-subalgebra
 of $\End^0(Z)$ with the same $1=1_Z$. Let $\Sigma_E$ be the $[E:\Q]$-element set of field embeddings $\sigma: E  \hookrightarrow \C$.
Let us define for each $\sigma\in \Sigma_E$ the corresponding {\sl weight subspace}
$$\Omega^1(Z)_{\sigma}=\{\omega \in \Omega^1(Z)\mid j_Z(e)\omega=\sigma(e)\omega \ \forall e \in E\}\subset \Omega^1(Z).$$
The well known splitting
$$E\otimes_{\Q}\C =\oplus_{\sigma\in \Sigma_E}E\otimes_{E,\sigma}\C=\oplus_{\sigma}\C_{\sigma} \ \text{ where  }\C_{\sigma}=E\otimes_{E,\sigma}\C=\C$$
implies that
$$\Omega^1(Z)=\oplus_{\sigma\in \Sigma_E}\Omega^1(Z)_{\sigma}.$$
Let us put
$$n_{\sigma}=\dim_{\C}(\Omega^1(Z)_{\sigma}).$$

Let $D$ be the centralizer of $E$ in $\End^0(Z)$.  Clearly, $E$ lies in the center of $D$, which makes $D$ is a finite-dimensional $E$-algebra.
It is also clear that each subspace  $\Omega^1(Z)_{\sigma}$ is $j_Z(D)$-invariant, which gives us a $\Q$-algebra homomorphism
\begin{equation}
\label{difSigma}
j_{Z,\sigma}:D\to \End_{\C}(\Omega^1(Z)_{\sigma}),
\end{equation}
that sends $1=1_Z\in D$ to the identity automorphism of the $\C$-vector space 
$\Omega^1(Z)_{\sigma}$.

\begin{lem}
\label{divOmega1}
Let $D$ be as above. Suppose that $D$ is a central simple $E$-algebra of dimension $d^2$ where $d$ is a positive integer. Then: 

\begin{itemize}
\item[(i)]
$d$ divides
all the multiplicities $n_{\sigma}$.
In particular, if $n_{\sigma}=1$ for some $\sigma \in \Sigma_E$ then $d=1$ and $D=E$.
\item[(ii)]
Let $M$ be the the number of $\sigma$'s with $n_{\sigma}\ne 0$.
Then 
$$dM \le  \dim(Z).$$
In particular, if $d=2\dim(Z)/[E:\Q]$ then $M\le [E:\Q]/2$. 
\end{itemize}
\end{lem}

\begin{proof}
Our condition on $D$ implies that 
$D_{\sigma}=D\otimes_{E,\sigma}\C$ is isomorphic to the matrix algebra $\Mat_d(\C)$ of size $d$ over $\C$ for all $\sigma \in \Sigma_E$.

We may assume that $n_{\sigma}>0$. Then $\Omega^1(Z)_{\sigma}\ne \{0\}$ and
$j_{Z,\sigma}(D) \ne \{0\}$. Extending $j_{Z,\sigma}$ by $\C$-linearity, we get a $\C$-algebra homomorphism
$$D_{\sigma} \to \End_{\C}((\Omega^1(Z)_{\sigma}),$$
which provides  $\Omega^1(Z)_{\sigma}$  with the structure of a $D_{\sigma}=\Mat_d(\C)$-module. This implies that each
$n_{\sigma}$
is divisible by $d$. This proves (i). In order to prove  (ii), it suffices to notice that
$$\dim(Z)=\sum_{\sigma}\dim_{\C}(\Omega^1(Z)_{\sigma})=\sum_{\sigma} n_{\sigma} \ge dM.$$

\end{proof}
\begin{rem}
 \label{ssLambda}
 \begin{itemize}
 \item[(i)]
 Let $\Lambda$ be the centralizer of $\delta_p$ in  $\End(J(C_{f,p}))$, which coincides with the centralizer
 of $\Z[\delta_p]$ in  $\End(J(C_{f,p}))$.
 Let us consider the $\Q$-subalgebra
 $$\Lambda_{\Q}=\Lambda\otimes\Q\subset \End(J(C_{f,p}))\otimes\Q=\End^0(J(C_{f,p})).$$
  Clearly, $\Lambda_{\Q}$ coincides with the centralizer of
$\Q[\delta_p]$ in $\End^0(J(C_{f,p}))$. Since $\delta_p$ respects the {\sl theta
 divisor} on the Jacobian $J(C_{f,p})$, the algebra $\Lambda_{\Q}$ is stable
under the corresponding {\sl Rosati involution} and therefore is
semisimple as a $\Q$-algebra. It is also clear that  the number field $\Q[\delta_p]\cong \Q(\zeta_p)$ lies in the center of $\Lambda_{\Q}$.
Hence, $\Lambda_{\Q}$ becomes a semisimple $\Q[\delta_p]$-algebra.
\item[(ii)]
Let $i$ be an integer such that $1 \le i \le p-1$. We write $\sigma_i$ for the field embedding
$$\sigma_i:\Q[\delta_p]\hookrightarrow \C$$
that sends $\delta_p$ to $\zeta_p^{-i}$. Let us consider the corresponding subspace $\Omega^1(J(C_{f,p}))_{\sigma_i}$ of differentials of the first kind on $J(C_{f,p})$.
It is known \cite[Remark 3.7]{ZarhinCambridge} that if $p$ does {\sl not} divide $n$ then
\begin{equation}
\label{niSigma}
n_{\sigma_i}=\dim_{\C}(\Omega^1(J(C_{f,p}))_{\sigma_i})=\left[\frac{ni}{p}\right].
\end{equation}
\item[(iii)] It follows from Lemma \ref{divOmega1} applied to $Z=J(C_{f,p})$ and $E=\Q[\delta_p]$ that
if $p$ does {\sl not} divide $n$ and
$\Lambda_{\Q}$ is a central simple $\Q[\delta_p]$-algebra of dimension $d^2$ then $d$ divides all $[ni/p]$ for all
integers $i$ with $1\le i \le p-1$.
\item[(iv)]    Suppose that either $n=p+1$,  or $n-1$ is {\sl not} divisible by $p$. Then the greatest common divisor of all $n_{\sigma_i}$'s is $1$
\cite[Lemma 8.1(D) on p. 516--517]{ZarhinGanita}.
 It follows that if $\Lambda_{\Q}$ is a central simple $\Q[\delta_p]$-algebra
then $\Lambda_{\Q}=\Q[\delta_p]$.
\item[(v)]    Suppose that $p$ divides $n-1$, say, $n=kp+1$ where $k$ is an integer.
Then the greatest common divisor of all $n_{\sigma_i}$'s is $k$.
\cite[Lemma 8.1(D) on p. 516--517]{ZarhinGanita}
 It follows that if $\Lambda_{\Q}$ is a central simple $\Q[\delta_p]$-algebra of dimension $d^2$ then $d$ divides $k$.

\item[(vi)]  The number of $i$ with $n_{\sigma_i}>0$ is at least $(p+1)/2$ \cite[p. 101]{ZarhinCrelle}.  It follows that if $\Lambda_{\Q}$ is a central simple $\Q[\delta_p]$-algebra of dimension $d^2$
then, in light of Proposition \ref{divOmega1},
$$d\cdot \frac{p+1}{2}\le g$$
where $g=\dim(J(C_{f,p}))$ is the genus of $C_{f,p}$.  This implies that
$$d\le \frac{2g}{p+1}< \frac{2g}{p-1}.$$
\end{itemize}
 \end{rem}

\begin{lem}
\label{redMatrix}
Let $\mathcal{H}$ be a finite-dimensional 
$E$-algebra, and $\Lambda$  an order in  $\mathcal{H}$ that contains $O_E$.
(In particular,  $\Lambda$ is a finitely generated torsion-free $O_E$-module and 
 the natural map
$\Lambda\otimes  \Q\to \mathcal{H}$ is an isomorphism of finite-dimensional $\Q$-algebras.)

Suppose that there are a positive integer $d$ and   a maximal ideal $\mathfrak{m}$  of $O_E$ with residue field $k=O_E/\mathfrak{m}$ such that the $k$-algebra
$\Lambda/\mathfrak{m}\Lambda$ is isomorphic to the matrix algebra $\Mat_d(k)$ of size $d$ over $k$.

Then $\mathcal{H}$ is a central simple $E$-algebra of dimension $d^2$.
\end{lem}

\begin{proof} Let $C_{\Q}$ the center of $\mathcal{H}$ that is a finite-dimensional commutative $E$-algebra.
Then $C:=C_{\Q}\cap \Lambda$ is the center of $\Lambda$.

Clearly, $C$  contains $O_E$ and is a saturated $O_E$-submodule of $\Lambda$. The latter means that
if $e u \in C$ for some $u\in \Lambda$ and   nonzero $e\in O_E$ then $u \in \Lambda$. This implies that the quotient $\Lambda/C$ is torsion-free (and finitely generated)
$O_E$-module and therefore is projective. 
It follows that $C$ is a direct summand of the $O_E$-module  $D$ and therefore there is an $O_E$-submodule $\mathfrak{P}$ of $\Lambda$ such that
$$\Lambda=C\oplus \mathfrak{P}.$$
Similarly, $O_E$ is a saturated $O_E$-submodule of $C$ and, by the same token, there is a locally free $O_E$-submodule $\mathfrak{Q}$ of $C$ such that
$$C=O_E\oplus \mathfrak{Q} \ \text{ and } \Lambda=C\oplus \mathfrak{P}=O_E\oplus \mathfrak{Q}\oplus  \mathfrak{P}.$$
Then the the natural map of $O_E/\mathfrak{m}=k$-modules,
$$O_E/ \mathfrak{m}\oplus  \mathfrak{Q}/ \mathfrak{m}   \mathfrak{Q} = C/\mathfrak{m}C   \to \Lambda/\mathfrak{m}\Lambda\cong \Mat_d(k)$$
 is {\sl injective}
and its image lies in the center $k$ of $\Mat_d(k)$.
The $k$-dimension arguments imply that  $\mathfrak{Q}/ \mathfrak{m}   \mathfrak{Q}=\{0\}$. Since $\mathfrak{Q}$ is finitely generated projective,
$\mathfrak{Q}=\{0\}$, i.e.,  
$C=O_E$ and the center of $\mathcal{H}$ is 
$$C_{\Q}=C\otimes \Q=O_E\otimes\Q=E.$$
Hence, $\mathcal{H}$ is a finite-dimensional
$E$-algebra with center $E$.

Let us check the simplicity of $\mathcal{H}$. Let $J_{\Q}$ be a proper two-sided ideal of $\mathcal{H}$. We need to check that $J_{\Q}=\{0\}$.
In order to do that, let us consider the intersection $J:=J_{\Q}\cap \Lambda$, which is obviously a two-sided ideal of $\Lambda$.
It is also clear that $J$ is a saturated $O_E$-submodule of $\Lambda$, i.e., the quotient $\Lambda/J$ is a torsion free (finitely generated)
$O_E$-module. Hence, $\Lambda/J$ is a projective $O_E$-module. It follows that $J$ is a direct summand of the $O_E$-module $\Lambda$,
i.e., there exists an $O_E$-submodule $\mathfrak{Q}$ of $\Lambda$ such that $\Lambda=J\oplus \mathfrak{Q}$. If $\mathfrak{Q}=\{0\}$
then $\Lambda=J$ and
$$\mathcal{H}=\Lambda\otimes \Q=J\otimes\Q=J_{\Q};$$
so $J_{\Q}=\mathcal{H}$, which is not true, because $J_{\Q}$ is a {\sl proper} ideal of $\mathcal{H}$.
This implies that
$\mathfrak{Q}\ne \{0\}$ and therefore $\mathfrak{Q}/\mathfrak{m}\mathfrak{Q}\ne \{0\}$. We have 
$$\Lambda/\mathfrak{m}\Lambda=J/\mathfrak{m}J \oplus \mathfrak{Q}/\mathfrak{m}\mathfrak{Q}.$$
Clearly, $J/\mathfrak{m}J$ is a proper two-sided ideal of the simple  algebra $\Lambda/\mathfrak{m}\Lambda\cong \Mat_d(k)$.
This implies that $J/\mathfrak{m}J=\{0\}$ , which implies that $J=\{0\}$ and therefore $J_{\Q}=\{0\}$.

To summarize: $\mathcal{H}$ is a simple finite-dimensional $E$-algebra with center $E$, i.e., a finite-dimensional central simple $E$-algebra.

On the other hand, the $E$-dimension of $\mathcal{H}$ equals the rank of the locally free $O_E$-module $\Lambda$, which, in turn,
equals the $k=O_E/\mathfrak{m}$-dimension of $\Lambda/\mathfrak{m}\Lambda$. Since
 $\Lambda/\mathfrak{m}\Lambda \cong \Mat_d(k)$ has $k$-dimension $d^2$, the $E$-dimension of $H$ is also $d^2$.
 It follows that $\mathcal{H}$ is a central simple $E$-algebra of dimension $d^2$.

\end{proof}

\section{Permutation groups and permutation modules}
\label{permute}
Our exposition in this section follows closely \cite[Sect. 2]{ZarhinCrelle}, see also \cite{Mortimer}.

Let $n \ge 5$ be an integer, $B$ a $n$-element set, and
 $\Perm(B)$  the group of permutations of $B$, which is isomorphic to the full symmetic group $\mathrm{S}_n$.
 The group  $\mathrm{S}_n$ has order $n!$ and contains precisely one (normal) subgroup of index $2$ that we denote by 
 $\Alt(B)$. Any isomorphism  between $\Perm(B)$ and $\mathrm{S}_n$ induces an isomorphism between $\Alt(B)$ and  the alternating group $\An$. 
 Since $n \ge 5$, the group  $\Alt(B)$ is simple non-abelian; its order is   $n!/2$.
Let $G$ be a  subgroup of $\Perm(B)$.

Let $F$ be a field. We write $F^B$
for the $n$-dimensional $F$-vector space of maps $h:B \to F$.
The space $F^B$ is  provided with a natural action of $\Perm(B)$ defined
as follows. Each $s \in \Perm(B)$ sends a map
 $h:B\to F$ into  $sh:b \mapsto h(s^{-1}(b))$. The permutation module $F^B$
contains the $\Perm(B)$-stable hyperplane
$$(F^B)^0=
\{h:B\to F\mid\sum_{b\in B}h(b)=0\}$$
and the $\Perm(B)$-invariant line $F \cdot 1_B$ where $1_B$ is the constant function $1$. The quotient $F^B/(F^B)^0$ is a trivial $1$-dimensional $\Perm(B)$-module.

Clearly, $(F^B)^0$ contains $F \cdot 1_B$ if and only if $\fchar(F)$ divides $n$.
 If this is {\sl not} the case then there is a $\Perm(B)$-invariant splitting
$$F^B=(F^B)^0 \oplus F \cdot 1_B.$$

Let $G$ be a subgroup of  $\Perm(B)$.
Clearly, $F^B$ and $(F^B)^0$  carry natural structures of  $G$-modules or (which is the same) of $F[G]$-modules. (Hereafter $F[G]$ stands for the {\sl group algebra} of $G$.)

If $F=\Q$ then
 the character of $\Q^B$ sends each $g \in G$ to the number of fixed points of $g$ in $B$ (\cite{SerreRep}, ex. 2.2, p.12); it takes on values in $\Z$ and called the
 {\sl permutation character} of $B$. Let us denote by
$\bphi=\bphi_B:G \to \Q$
 the character of $(\Q^B)^0$.

If $\fchar(F)=0$ then the  $F[G]$-module $(F^B)^0$ is absolutely simple
\footnote{Recall that a simple $F[G]$-module $V$ is called {\sl absolutely simple} if the centralizer of $G$ in $\End_F(V)$ coincides with $F$ or equivalently
the natural homomorphism $F[G] \to \End_F(V)$ of $F$-algebras is surjective.}
 if and only if the action of  $G$ on $B$  is doubly transitive (\cite[ex. 2.6, p. 17]{SerreRep}, \cite{Mortimer}).
(Notice that $1+\bphi$ is the permutation character. This implies that the character $\bphi$ also takes on values in $\Z$.) In particular,   $\Q_p[G]$-module $(\Q_p^B)^0$ is absolutely simple if and only if the action of  $G$ on $B$  is doubly transitive.

In what follows we concentrate on the case of $F=\F_p$.

\begin{rem}
\label{DoubleprimeToP}
\begin{itemize}
\item
Let $p$ be a prime that does {\sl not} divide the order of $G$. This condition is automatically fulfilled if $p>n$, because $G$, being isomorphic to a subgroup of $\mathbf{S}_n$, has order that divides $n!$.  
\item
Suppose that the action of  $G$ on $B$  is {\sl doubly transitive}. 
Taking into account that the representation theory  of $G$ over $\Q_p$ is ``the same over $\F_p$ as over $\Q_p$'' (\cite[Sect. 15.5, Prop.43]{SerreRep}, \cite{Mortimer}), we conclude that the
$\F_p[G]$-module $(\F_p^B)^0$ is {\sl absolutely simple} (see also \cite[Cor. 7.5 on p. 513]{ZarhinGanita}).
\end{itemize}
\end{rem}

\begin{defn}
\label{heartDef}
Let $G$ be a subgroup of $\Perm(B)$.

If  $p\mid n$  then let us define the $G$-module
$$\left(\F_p^B\right)^{00}:=(\F_p^B)^0/(\F_p \cdot 1_B).$$
If $p$ does not divide $n$ then let us put
$$\left(\F_p^B\right)^{00} := (\F_p^B)^0.$$
The $G$-module $(\F^B)^0$
is called the {\sl heart} of the permutation  representation of $G$ on $B$ \cite{Mortimer}.
It follows from the definition that
$\dim_{\F_p}(\left(\F_p^B)^{00}\right))=n-1$ if $n$ is not divisible by $p$
 and $\dim_{\F_p}(\left(\F_p^B)^{00}\right)=n-2$ if $p\mid n$.
\end{defn}

\begin{lem}
\label{Anp}
Assume that $G=\Perm(B)$ or $\Alt(B)$.
Then the $G$-module $\left(\F_p^{B}\right)^{00}$ is absolutely simple.
\end{lem}

\begin{proof}
 This  result is well known (and goes back to Dickson).  See   \cite[Th. 5.2 on p. 133]{Glasgow}, \cite{Wagner}, \cite{Mortimer}, \cite[Lemma 2.6]{ZarhinCrelle}.
\end{proof}

\begin{rem}
\label{A5notVery}
It turns out that the case of $n=5$ and 
$$G=\Alt(B)\cong \mathbf{A}_5\cong \mathrm{PSL}(2,\F_5)=\mathrm{SL}(2,\F_5)/\{\pm 1\}$$ 
is rather special when
\begin{equation}
\label{pMod5} 
p \equiv \pm 1 \bmod 5.
\end{equation}
 Namely,  in this case $p>5$ and 
the $G=\mathrm{PSL}(2,\F_5)$-module $(\F_p^{B})^{00}=(\F_p^{B})^0$  viewed as the $\SL(2, \F_5)$-module splits into a nontrivial tensor 
product. In order to see this, recall \cite[Sect. 38]{Dornhoff} that $\mathrm{SL}(2,\F_5)$ has the ordinary character $\theta_2$ of degree $4$ (which, is the lift
of $\phi_5$  from $\mathbf{A}_5$)
and
 two {\sl ordinary irreducible} characters $\eta_1$ and $\eta_2$ of degree $2$ with
$$\Q(\eta_1)=\Q(\eta_2)=\Q(\sqrt{5}),$$
whose product $\eta_1 \eta_2$ coincides with $\theta_2$.
By the quadratic reciprocity law,
the congruence \eqref{pMod5}  implies that $\sqrt{5} \in \F_p$ and therefore $\sqrt{5}$ lies in the field $\Q_p$ of $p$-adic numbers,
because $p \ne 2,5$ is {\sl odd}. This means that
$$\Q_p(\eta_1)=\Q_p(\eta_2)=\Q_p.$$
By a theorem of Janusz \cite[Theorem (d) on p. 3-4]{Janusz}, characters of both $\eta_1$  and $\eta_2$ can be realized over $\Q_p$,
i.e., there are two-dimensional $\Q_p$-vector spaces $V_1$ and $V_2$ and
linear representations
\begin{equation}
\label{rho1rho2}
\rho_1: \mathrm{SL}(2,\F_5) \to \Aut_{\Q_p}(V_1) \cong \GL(2,\Q_p), 
\end{equation}
$$ \rho_2: \mathrm{SL}(2,\F_5) \to \Aut_{\Q_p}(V_2)\cong \GL(2,\Q_p),$$
whose {\sl characters} are $\eta_1$ and $\eta_2$ respectively.  Let $T_1$ and $T_2$ be any $\mathrm{SL}(2,\F_5)$-invariant $\Z_p$-lattices of rank $2$ in $V_1$ and $V_2$ respectively.
Since the order $120$ of the group $\mathrm{SL}(2,\F_5)$ is prime to $p$ and  the $\Q_p[\mathrm{SL}(2,\F_5)]$-modules $V_1$ and $V_2$ are simple, 
it follows from  \cite[Sect. 15.5, Prop. 43]{SerreRep}) that
their {\sl reductions} modulo $p$
$$\bar{V}_1=T_1/pT_1, \ \bar{V}_2=T_2/pT_2$$
are simple $\F_p[\mathrm{SL}(2,\F_5)]$-modules.
On the other hand, the tensor product
$$T:=T_1 \otimes_{\Z_p} T_2\subset V_1 \otimes_{\Q_p} V_2$$
 is a $\mathrm{SL}(2,\F_5)$ -invariant $\Z_p$-lattice of rank $4$ in $V_1\otimes_{\Q_p}V_2=:V$.
 The equality 
 \begin{equation}
 \label{eta1eta2phi}
 \eta_1 \eta_2=\theta_2
 \end{equation}
   of the corresponding class functions on  $\mathrm{SL}(2,\F_5)$ implies (if we take into account that $\phi_B$ is irreducible)
  that the $\Q_p[\mathrm{SL}(2,\F_5)]$-module
 $V$ is simple and the $\F_p[\mathrm{SL}(2,\F_5)]$-module
 \begin{equation}
 \label{V1tensorV2}
 T/pT=\big(T_1 \otimes_{\Z_p} T_2\big)/p= (T_1/pT_1)\otimes_{\F_p}(T_1/pT_1)=\bar{V}_1\otimes_{\F_p}\bar{V}_2
 \end{equation}
 is simple.   On the other hand, the equality \eqref{eta1eta2phi} implies the existence of an isomorphism
 $$u: V=V_1 \otimes_{\Q_p} V_2 \cong (\Q_p^{B})^0$$
 of the $\Q_p[\mathrm{SL}(2,\F_5)]$-modules. 
 
 Obviously, 
 $$(\Z_p^{B})^0:=\{h: B \to \Z_p\mid \sum_{b\in B}h(b)=0\}$$
 is  a $\mathrm{SL}(2,\F_5)$-invariant $\Z_p$-lattice of rank $4$ in $(\Q_p^{B})^0$.
 (Here $\mathrm{SL}(2,\F_5)$ acts on $(\Q_p^{B})^0$ through the quotient $\mathrm{SL}(2,\F_5)/\{\pm 1\}=\mathbf{A}_5$.)
 Notice that $u(T)$ is a (may be, another)   $\mathrm{SL}(2,\F_5)$-invariant $\Z_p$-lattice of rank $4$ in $(\Q_p^{B})^0$ and
 the $\F_p[\mathrm{SL}(2,\F_5)]$-module $u(T)/p \ u(T)$ is obviously isomorphic to $T/pT$.  In light of \cite[Sect. 15.1,  Th. 32]{SerreRep}, the {\sl simplicity} of  the $\F_p[\mathrm{SL}(2,\F_5)]$-modules
 $T/pT$ (and, hence, of $u(T)/p \ u(T)$) implies that the  $\F_p[\mathrm{SL}(2,\F_5)]$-modules $T/pT$ and $(\Z_p^{B})^0/p (\Z_p^{B})^0$ are isomorphic.
 Taking into account  \eqref{V1tensorV2} and  that $(\Z_p^{B})^0/p \cdot (\Z_p^{B})^0=(\F_p^{B})^0$, we conclude that
  that the $\mathrm{SL}(2,\F_5)$-modules
 $\bar{V}_1\otimes_{\F_p}\bar{V}_2$ and $(\F_p^{B})^0$ are isomorphic.
\end{rem}

\begin{rem}
One may find an explicit construction of the group embeddings $\SL(2,\F_5) \to \GL(2,\F_p)$ (when $p$ satisfies \eqref{pMod5})
in the book of M. Suzuki \cite[Ch.  3, Sect. 6]{Suzuki}.

\end{rem}

\section{Very simple and central simple representations}
\label{verySR}
\begin{defn}
Let $V$ be a vector space of over a field $F$, let $G$ be a group and
$\rho: G \to \Aut_F(V)$ a linear representation of $G$ in $V$.
Let $R \subset \End_F(V)$ be a $F$-subalgebra containing the
identity map $$\Id: V \to V.$$
\begin{itemize}
\item[(i)]
We say that $R$ is $G$-normal if
 $$\rho(\sigma) R \rho(\sigma)^{-1} \subset R \quad \forall
 \sigma \in G.$$
 \item[(ii)]
 We say that a normal $G$-subalgebra is {\sl obvious} if it coincides either
 with $F \cdot  \Id$ or with $\End_F(V)$.
  \item[(iii)]
 We say that  the $G$-module $V$ is {\sl very simple} if 
 every  $G$-normal subalgebra of $\End_F(V)$ is obvious.
  \item[(iii)]
 We say that  the $G$-module $V$ is {\sl cental simple} if 
 every  $G$-normal subalgebra of $\End_F(V)$  is a central simple $F$-algebra.
   \item[(iv)]
 We say that  the $G$-module $V$ is {\sl strongly simple} if 
 every  $G$-normal subalgebra of $\End_F(V)$  is a  simple $F$-algebra.
 \end{itemize}
\end{defn}

\begin{rem}
\label{image}
\begin{enumerate}
\item[(i)]
Clearly, a very simple $G$-module is central simple and strongly simple. It is also clear that
a central simple $G$-module is strongly simple.

\item[(ii)]
 Clearly,  a subalgebra $R \subset \End_F(V)$ is $G$-normal  if and only if it is $\rho(G)$-normal. It follows readily that
 the $G$-module $V$ is very simple (resp. central simple) (resp. strongly simple) if and only if the corresponding $\rho(G)$-module $V$ is very simple (resp. central simple)  (resp. strongly simple). 
 It is known \cite[ Rem. 2.2(ii)]{ZarhinM} that a very simple module is absolutely simple.
 \item[(iii)] If $R$ is a $G$-normal subalgebra of $\End_F(V)$ then
$$\rho(\sigma) R \rho(\sigma)^{-1} = R \quad \forall
 \sigma \in G.$$
 Indeed, suppose that there is $u \in R$ such that for some $\sigma \in G$
 $$u \not\in \rho(\sigma) R \rho(\sigma)^{-1}.$$
 This implies that
 $$ \rho(\sigma^{-1}) u  \rho(\sigma^{-1})^{-1}=
 \rho(\sigma)^{-1}u \rho(\sigma)
  \not\in  \rho(\sigma)^{-1}\big(\rho(\sigma) R \rho(\sigma)^{-1}\big)\rho(\sigma)=R.$$
  It follows that
  $$ \rho(\sigma^{-1}) R  \rho(\sigma^{-1})^{-1}\not\subset R,$$
  which contradicts the normality of $R$, because $\sigma^{-1}\in G$. (Of course, if $\dim_F(V)$ is finite, the desired equality follows readily from the coincidence of $F$-dimensions of $R$ and 
 $\rho(\sigma) R \rho(\sigma)^{-1}$.)
 \item[(iv)]
If $G^{\prime}$ is a subgroup of $G$ then every $G$-normal subalgebra is also a normal $G^{\prime}$-subalgebra. It follows that
 if the $G^{\prime}$-module $V$ is very simple then the $G$-module $V$ is also very simple.
 \item[(v)]
Let us check that a strongly simple $G$-module $V$ is simple. Indeed, if it is not then there is a {\sl proper} $G$-invariant $F$-vector subspace $W$ of $V$. Then
 the $F$-subalgebra
 $$R:=\{u \in \End_F(V)\mid u(W) \subset W\}$$
 is $G$-normal but  even {\sl not semisimple}, because it contains a proper two-sided ideal
 $$I(W,V):= \{u \in \End_F(V)\mid u(V) \subset W\}.$$
 This proves the simplicity of $V$.  
 
The centralizer $\End_G(V)$ is obviously $G$-normal. This implies that it is a division algebra over $F$.
(Actually, it follows from the simplicity of the $G$-module $V$.
 
 If the $G$-module $V$ is central simple (resp.  very simple) then  normal $\End_G(V)$ is  a central division $F$-algebra
 (resp. coincides with $F \cdot \Id$).
 
  \item[(vi]) If $R$ is a $G$-normal subalgebra of $\End_F(V)$ then for each $\sigma\in G$ the map
  $$R \to R, \  u \mapsto \rho(\sigma)u \rho(\sigma)^{-1}$$
  is an {\sl automorphism} of the $F$-algebra $R$ (in light of (iii)).  This implies that if $C$ is the center of $R$ then
  $$\rho(\sigma)C \rho(\sigma)^{-1}=C$$
  for all $\sigma \in G$. This means that $C$ is a $G$-normal subalgebra of $\End_F(V)$.
 
 \end{enumerate}
 
\end{rem}

Recall that a module $V$ over a ring $R$ is called {\sl isotypic} if either $V$ is simple or is isomorphic to direct sum of finitely many copies
of a simple $R$-module $W$.
The following assertion is contained in \cite[Lemma 7.4]{ZarhinM}

\begin{lem}
\label{74}
Let $H$ be a group, $F$ a field, $V$ a vector space  of finite positive dimension $N$ over $F$.
Let $\rho: H \to \Aut_F(V)$ be an irreducible linear representation of $H$.  Let $R$ be a $H$-normal subalgebra of $\End_F(V)$. Then:
\begin{itemize}
\item[(i)]
The faithful $R$-module $V$ is semisimple.
\item[(ii)]
Either the $R$-module $V$ is isotypic or there is a subgroup $H^{\prime}$ of finite index $r$ in $H$ such that $r>1$ and $r$ divides $N$.
\end{itemize}
\end{lem}

\begin{prop}
\label{centralsimpleIndex}
Let $F$ be  a field, whose Brauer group $\mathrm{Br}(F)=\{0\}$. (E.g., $F$ is either finite or an algebraically closed field.)
Let $V$ be a vector space of finite positive dimension $N$ over $F$. Let $H$ be a group and
$\rho: H \to \Aut_F(V)$ a linear absolutely irreducible representation of $H$ in $V$. Suppose that 
every maximal subgroup of $H$ has index that does not divide $N$.


Then the $H$-module $V$ is central simple.  
\end{prop}

\begin{proof}
Slightly abusing the notation, we write $F$ instead of $F\cdot \Id$.

Let $R$ be a $H$-normal subalgebra of $\End_F(V)$. It follows from Lemma \ref{74} that the faithful $R$-module $V$ is {\sl isotypic}, i.e., there is a simple faithful  $R$-module $W$  and a positive integer $a$ such that the $R$-modules $V$ and $W^a$ are isomorphic. The existence of a faithful simple $R$-module implies that $R$ is a simple $F$-algebra. In particular, the center $k$ of $R$ is a field. We have
$$F=F\cdot \Id\subset k\subset R\subset \End_F(V).$$
Then $V$ carries the natural structure of a $F$-vector space. This implies that the degree
$[k:F]$ divides $\dim_F(V)=N$.

The center $k$ of $H$-normal $R$ is also $H$-normal (see Remark \ref{image}(v)). This gives rise to the group homomorphism
\begin{equation}
\label{centerConj}
H \to \Aut(k/F),  \ \sigma \mapsto \{c \mapsto \rho(\sigma)u \rho(\sigma)^{-1}\}.
\end{equation}
Here $\Aut(k/F)$ is the automorphism group of the field extension $k/F$. By Galois theory,
the order of $ \Aut(k/F)$ divides $[k: F]$, which in turn, divides $N$.  This implies
that the kernel of the homomorphism \eqref{centerConj} is a subgroup of $H$, whose index divides $N$.
Our condition on indices of subgroups of $H$ implies that the kernel coincides with the whole $H$,
i.e., the homomorphism \eqref{centerConj} is trivial. This means that all elements of $k$ commute with  $\rho(\sigma)$
for all $\sigma\in H$.  The absolute irreducibility of $\rho$ implies that $k\subset F$ and therefore
$$k= F=F\cdot \Id.$$
So, $R$ is a simple $F$-algebra with center $ F\cdot \Id$, i.e., is a central simple $F$-algebra. This ends the proof.
\end{proof}

\begin{thm}
\label{AdGH}
Let $F$ be a field, whose Brauer group $\mathrm{Br}(F)=\{0\}$. (E.g., $F$ is either finite or an algebraically closed field.)
Let
$V$ be an $F$-vector space of finite  dimension $N>1$.
Let $G$ be a group and
$$\rho: G \to \Aut_F(V)$$
be a group homomorphism.
Let $H$ be a  normal subgroup of  $G$ that enjoys the following properties.
\begin{itemize}
\item[(i)]
If $H^{\prime}$ is a subgroup of $H$ of finite index $N^{\prime}$ and $N^{\prime}$ divides $N$
then $H^{\prime}=H$.
\item[(ii)]
$H$ is a simple non-abelian group. Assume additionally that either $H=G$,  or $H$  is the only proper normal subgroup of $G$.
\item[(iii)]
The $H$-module $V$ is absolutely simple, i.e., the representation of $H$ in $V$ is irreducible and the centralizer
$\End_H(V)=F \cdot \Id$.
\end{itemize}
Let $R \subset \End_F(V)$ be a $G$-normal subalgebra.   Then there are positive integers $a$ and $b$ that enjoy the following properties.
\begin{enumerate}
\item[(a)]
$N=ab$;
\item[(b)]
The $F$-algebra $R$ is isomorphic to the matrix algebra $\Mat_a(F)$ of size $a$ over $F$.
In particular, the $G$-module $V$ is central simple.
\item[(c)]
The $R$-module  $V$ is semisimple, isotypic and isomorphic to $R^b$.
In addition, 
the centralizer $\tilde{R}=\End_R(V)$ is a normal $G$-subalgebra that is isomorphic to  the matrix algebra $\Mat_b(F)$ of size $b$ over $F$.
\item[(d)]
Suppose that $a \ne 1, b \ne 1$ (i.e., $R$ is not obvious).  Then both homomorphisms
$$ \mathrm{Ad}_R: G \to \Aut(R)=R^{*}/F^{*}\Id\cong \mathrm{GL}(a,F)/F^{*}=\mathrm{PGL}(a,F),$$
$$\mathrm{Ad}_R(\sigma) (u)=\rho(\sigma)u \rho(\sigma)^{-1} \ \forall u \in R$$
and
$$ \mathrm{Ad}_{\tilde{R}}: G \to \Aut(\tilde{R})=\tilde{R}^{*}/F^{*}\Id\cong \mathrm{GL}(b,F)/F^{*}=\mathrm{PGL}(b,F),$$
$$\mathrm{Ad}_{\tilde{R}}(\sigma) (u)=\rho(\sigma)u \rho(\sigma)^{-1} \ \forall u \in \tilde{R}$$
(with $\sigma \in G$) are injective. In addition,
$$ \mathrm{Ad}_R(H)\subset \mathrm{PSL}(a,F), \ \mathrm{Ad}_{\tilde{R}}(H)\subset \mathrm{PSL}(b,F).$$
\item[(e)]
The $H$-module $V$ is central simple.
\end{enumerate}
\end{thm}

\begin{proof}
{\bf Step 0}.
Since $H$ is a simple group, $V$ is a faithful $H$-module. In light of (ii), $V$ is a faithful $G$-module.

Clearly, $V$ is a faithful $R$-module. Since $R$ is $G$-normal,
\begin{equation}
\label{NormalEq}
\rho(\sigma) R \rho(\sigma)^{-1} = R \quad
\forall \sigma \in G.  
\end{equation}
It follows from (ii) that 
$$H=[H,H]\subset [G,G]\subset G$$
and either $G=H$ or $G/H$ is a finite simple group (e.g., a cyclic group of prime order).

{\bf Step 1}. By Lemma \ref{74}(i), $V$ is a {\sl semisimple} $R$-module.

{\bf Step 2}.   In light of Lemma \ref{74}(ii), property (i) implies that
the $R$-module $V$ is {\sl isotypic}.

{\bf Step 3}. Since the faithful $R$-module $V$ is an isotypic, there exist a faithful simple
$R$-module $W$   and a positive integer $b$ such that $V \cong W^b$. If we put $a=\dim_F(W)$ then we get
    $$ba=b \cdot \dim_F(W)=\dim_F(V)=N.$$
Clearly, $\End_R(V)$ is isomorphic to the matrix algebra
$\Mat_b(\End_R(W))$  of size $b$ over $\End_R(W)$.

Consider the centralizer
    $$k:=\End_R(W)$$
    of $R$ in $\End_F(W)$.
Since $W$ is a simple $R$-module, $k$ is a finite-dimensional division algebra over $F$. Since $\mathrm{Br}(F)=\{0\}$, $k$ must be a field.
Hence, the automorphism group $\Aut_F(k)$ of the $F$-algebra $k$ is actually the automorphism group $\Aut(k/F)$
 of the field extension $k/F$. It follows that  $\Aut_F(k)=\Aut(k/F)$ is  finite and its order divides the degree $[k:F]$.
We have
    $$\tilde{R}=\End_R(V) \cong \Mat_b(k).$$
Clearly,  the $F$-subalgebra $\tilde{R}=\End_R(V) \subset \End_{F}(V)$ is stable under the ``adjoint action'' of $G$, which gives rise to the group homomorphism
$$\mathrm{Ad}_{\tilde{R}}: G \to\Aut(\tilde{R}).$$
 Since $k$ is the center of $\Mat_b(k)$, it is stable under the action of $G$ and of its subgroup $H$. This gives rise to the group homomorphism 
 $$H\to\Aut(k/F), \  h \mapsto \{\lambda \mapsto  \rho(h)\lambda \rho(h)^{-1} \ \forall  \lambda\in k\} \ \forall h \in H,$$ whose 
  kernel $H^{\prime}$ has index $[H:H^{\prime}]$ dividing  $[k:F]$. Since $V$ carries the natural structure of a $k$-vector space, 
  $[k:F]$ divides $\dim_F(V)=N$, the index  $[H:H^{\prime}]$ divides $N$. In light of (i), $H^{\prime}=H$, i.e., the homomorphism is trivial.
  This means that
center $k$ of $\End_R(V)$ commutes with $\rho(H)$. Since $\End_H(V)=F$, we have $k=F$. This implies that
$\End_R(V) \cong \Mat_b(F)$ and
$$\mathrm{Ad}_{\tilde{R}}: G \to \Aut_F(\tilde{R})=\tilde{R}^{*}/F^{*}\Id \cong\GL(b,F)/F^*=\PGL(b,F)$$
 kills $H$ if and only if
$\tilde{R}=\End_R(V) \subset \End_H(V)=F \cdot \Id$.
 Since $\tilde{R}=\End_R(V) \cong \Mat_b(F)$, the homomorphism $\mathrm{Ad}_{\tilde{R}}$ kills $H$  if and only if $b=1$, i.e., $V$ is an absolutely simple (faithful) $R$-module.
 This means that if $b>1$ then   $\mathrm{Ad}_{\tilde{R}}$ does {\sl not} kill $H$, i.e., the normal subgroup $\ker(\mathrm{Ad}_{\tilde{R}})$ of $G$ does {\sl not} contain $H$.
  In light of (ii), this implies that the group homomorphism
 $$\mathrm{Ad}_{\tilde{R}}: G \to \Aut(\tilde{R}) \cong \PGL(b,F)$$
 is {\sl injective} if $b>1$.

Since $V$ is a semisimple module over the subalgebra $R$ of $\End_F(V)$ and $\tilde{R}$ is the centralizer of $R$ in $\End_F(V)$,
 it follows from the Jacobson density theorem that
 $$R=\End_{\tilde{R}}(V) \cong \End_F(W) \cong \Mat_a(F).$$
The ``adjoint action''   of $G$ on $R$ gives rise to the homomorphism
    $$\mathrm{Ad}_R: G  \to  \Aut(R)=R^{*}/F^{*}\Id \cong \PGL(a,F).$$
 Clearly, $\mathrm{Ad}_R$ kills $H$ if and only if $R$ commutes with $\rho(H)$, i.e., $R=F \cdot\Id$, which is equivalent to the equality $a=1$.
 This means that if $a>1$ then   $\mathrm{Ad}_R$ does {\sl not} kill $H$, i.e., the normal subgroup $\ker(\mathrm{Ad}_{R})$ of $G$ 
 does {\sl not} contain $H$.
  In light of (ii), this implies that the group homomorphism
 $$\mathrm{Ad}_R: G \to \Aut(\tilde{R}) \cong \PGL(a,F)$$
 is {\sl injective} if $a>1$.
 
 To summarize: a normal $G$-subalgebra $R$ is {\sl not} obvious if  and only if 
 $$a>1,b>1.$$
  If this is the case then both group homomorphisms
 $$\mathrm{Ad}_R: G \to \PGL(a,F), \  \mathrm{Ad}_{\tilde{R}}: G \to \PGL(b,F)$$
 are {\sl injective}.  
 
 The last assertions of  Theorem \ref{AdGH}(d) about the images of $H$ follow from the  equality $H=[H,H]$ and the inclusions
 $$[\mathrm{GL}(a,F), \mathrm{GL}(a,F)]\subset \mathrm{SL}(a,F), \ [\mathrm{GL}(b,F), \mathrm{GL}(b,F)]\subset \mathrm{SL}(b,F).$$
 The assertion (e) follows readily from the second assertion of (b) (if we replace $G$ by $H$).
 This ends the proof.
\end{proof}

\begin{cor}
\label{AdGHcor}
Keeping the assumption and notation of Theorem \ref{AdGH}, assume additionally that $N=2\ell$ where $\ell$ is a prime.
If the $H$-module $V$ is not very simple then there exist  group embeddings
$$G \hookrightarrow \mathrm{PGL}(2,F), \ H \hookrightarrow \mathrm{PSL}(2,F).$$
\end{cor}

\begin{proof}
Let $R$ be a $H$-normal non-obvious $H$-subalgebra  and $\tilde{R}=\End_R(V)$.  By Theorem \ref{AdGHcor},
there are positive integers $a$ and $b$ such that
$$ab=N, a>1, b>1; R \cong \Mat_a(F), \tilde{R} \cong \Mat_b(F).$$
Our conditions on $N$ imply  that either $a=2, b=\ell$ or $a=\ell, b=2$. By Theorem \ref{AdGH}, there are group embeddings
$$G \hookrightarrow \mathrm{PGL}(a,F), \ H \hookrightarrow \mathrm{PSL}(a,F)$$
and
$$G \hookrightarrow \mathrm{PGL}(b,F), \ H \hookrightarrow \mathrm{PSL}(b,F).$$
Since either $a$ or $b$ is $2$, there are group embeddings 
$$G \hookrightarrow \mathrm{PGL}(2,F), \ H \hookrightarrow \mathrm{PSL}(2,F).$$
\end{proof}

\begin{thm}
\label{A5}
Suppose that $n \ge 5$ is an integer, 
 $B$ is an $n$-element set, and $p$ is a prime.  
 Let us consider the vector space  $\left(\F_p^{B}\right)^{00}$ over the field $\F_p$ endowed
 with the natural structure of a $ \Perm(B) $-module  (see Definition \ref{heartDef}),
and let 
$$\rho: \Perm(B) \to \Aut_{\F_p}\left(\left(\F_p^{B}\right)^{00}\right)$$
be the  corresponding structure homomorphism.

 Then:
 \begin{itemize}
 \item[(i)]
The $\Perm(B)$-module $\left(\F_p^{B}\right)^{00}$ is very simple.
 \item[(ii)]
The $\Alt(B)$-module $\left(\F_p^{B}\right)^{00}$ is very simple if and only if  either $n>5$ or
$$n=5,  \ p \not\equiv \pm 1 \ \bmod 5.$$
\item[(iii)] Suppose that 
$$n=5, \ p \equiv \pm 1 \ \bmod 5.$$ 
and $R \subset \End_{\F_p}(\left(\F_p^{B}\right)^{00})$ is a $\Alt(B)$-normal subalgebra.

 Then either $R=\F_p\cdot \Id$, or $R=\End_{F_p}(\left(\F_p^{B}\right)^{00})$, or the $\F_p$-algebra $R$ is isomorphic to the matrix algebra 
 $\Mat_2(\F_p)$ of size 2 over $\F_p$.
 \item[(iv)]
 The $\Alt(B)$-module $\left(\F_p^{B}\right)^{00}$ is central  simple
\end{itemize}
\end{thm}

\begin{rem}
\label{DoneN}

The assertion of Theorem \ref{A5} was earlier proven in the following cases.
\begin{itemize}
\item[(A)]
$p\in \{2,3\}$, see \cite[Ex. 7.2]{ZarhinM} and \cite[Cor. 4.3]{ZarhinCrelle}.
\item[(B)]
$p>3$ and $n \ge 8$, see \cite[Cor. 4.6]{ZarhinCrelle}.
\item[(C)]
$N=\dim_{\F_p}(\left(\F_p^{B}\right)^{00})$ is a prime. It follows readily from  \cite[Cor. 4.4(i)]{ZarhinCrelle} applied to $H=\mathrm{Alt}(B)$ and $V=(\F_p^{B})^{00}$.
(We may apply this result from \cite{ZarhinCrelle}, because  $\Alt(B)$ is a simple non-abelian group of order $n!/2$ and therefore its order is bigger that the order
of $\mathbf{S}_N$, since $N \le n-1$.)
\end{itemize}
 
So, In the course of the proof we may assume that 
\begin{equation}
\label{n567}
p>3; \  n\in \{5,6,7\}.
\end{equation}
\end{rem}

\begin{proof}[Proof of Theorem \ref{A5}]
We assume that \eqref{n567} holds.

{\bf Step 1}. First assume  that $p\mid n$. Then either $n=p=5$ or $n=p=7$. In both
cases $$N=\dim_{\F_p}((\F_p^{B})^{00})=n-2$$ is a prime. 
Now the
very simplicity of  the $\mathrm{Alt}(B)$-module $(\F_p^{B})^{00}$ follows from  Remark \ref{DoneN}(A).
So, we may assume that $p$ does {\sl not} divide $n$ and therefore 
$$N=n-1.$$

{\bf Step 2}. If $n=6$ then $N=5$ and the
very simplicity of  the $\mathrm{Alt}(B)$-module $(\F_p^{B})^{00}$ follows from  Remark \ref{DoneN}(C).
So, we may assume that 
$$n \in \{5,7\}.$$

{\bf Step 3}. Suppose that $n=7$. Then $N=6 =2\times 3$ where  $3$ is a prime. 
It follows from Corollary  \ref{AdGHcor} that if  the $\mathrm{Alt}(B)$-module $(\F_p^{B})^{00}$ is {\sl not} very simple
then there is a  group homomorphism
$$\mathrm{Ad}_R: \mathrm{Alt}(B) \hookrightarrow \mathrm{PSL}(2,\F_p).$$
However,  it is known \cite[Th. 6.25 on p. 412 and Th. 626 on p. 414]{Suzuki} that $ \mathrm{PSL}(2,\F_p)$
does {\sl not} contain a subgroup isomorphic to $\mathbf{A}_7 $.  Since $\mathrm{Alt}(B) \cong \mathbf{A}_7$, we get a contradiction, which implies that
the $\mathrm{Alt}(B)$-module $(\F_p^{B})^{00}$ is  very simple if $n=7$.

{\bf Step 4}. Suppose that $n=5$.  We are going to apply  Corollary  \ref{AdGHcor} to $H=\Alt(B)$, $G=\Perm(B)$ or $\Alt(B)$, and $V=(\F_p^{B})^{00}$.

Since $p$ does {\sl not} divide $n=5$, we get $p > 5$, and $n-1=4=2 \times 2$ where $2$ is a prime. 

\begin{itemize}
\item
Suppose that the $\mathrm{Perm}(B)$-module $(\F_p^{B})^{00}$ is {\sl not} very simple. It follows from Corollary  \ref{AdGHcor} (applied to $G=\Perm(B), H=\Alt(B)$)
that there is a group embedding
$$\mathrm{Ad}_R: \mathrm{Perm}(B) \hookrightarrow \mathrm{PGL}(2,\F_p).$$
This implies that $\mathrm{PGL}(2,\F_p)$ contains a subgroup isomorphic to $\mathrm{S}_5$, because  $\mathrm{Perm}(B)\cong \mathbf{S}_5$.
Since $p>5$, the order $120$ of $\mathbf{S}_5$ is {\sl not} divisible by $p$. However, there are no finite subgroups of $\mathrm{PGL}(2,\F_p)$
that are isomorphic to $\mathbf{S}_5$ \cite[Th. 6.25 on p. 412 and Th. 626 on p. 414]{Suzuki}; see  also \cite[Sect. 2.5]{Serre1972}. The obtained contradiction proves that 
the $\mathrm{Perm}(B)$-module $(\F_p^{B})^{00}$ is  very simple if $n=5$.

\item
It follows from Corollary  \ref{AdGHcor} (applied to $G=H=\Alt(B)$ and $V=(\F_p^{B})^{00}$) that if  the $\mathrm{Alt}(B)$-module $(\F_p^{B})^{00}$ is {\sl not} very simple
then there  is an injective  group homomorphism
$$ \mathrm{Ad}_R: \mathrm{Alt}(B) \hookrightarrow \mathrm{PSL}(2,\F_p).$$
Then the order $60$ of the group $\mathrm{Alt}(B)$ divides the order $(p^2-1)p/2$ of the group $\mathrm{PSL}(2,\F_p)$. This implies that $5$ divides $p^2-1=(p+1)(p-1)$, i.e.,
$p \equiv \pm 1 \ \bmod 5$. This implies that if $n=5$ and $p \not\equiv \pm 1 \ \bmod 5$ then $\mathrm{Alt}(B)$-module $(\F_p^{B})^{00}$ is  very simple.
\end{itemize}

{\bf Step 5}. Suppose that $n=5$ and $p \equiv \pm 1 \ \bmod 5$. Let us prove that the $\mathrm{Alt}(B)$-module $(\F_p^{B})^{00}=(\F_p^{B})^{0}$ is {\sl not} very simple.
Recall (Remark \ref{Anp})  that there is a surjective homomorphism  $\SL(2,\F_5) \twoheadrightarrow \mathbf{A}_5$, and there are $\SL(2,\F_5) $-modules $\bar{V}_1$ and $\bar{V}_2$
with
$$\dim_{\F_p}(\bar{V}_1)=\dim_{\F_p}(\bar{V}_2)=2,$$
and an isomorphism of $\SL(2,\F_5) $-modules $(\F_p^{B})^{0}\cong \bar{V_1}\otimes_{\F_p}\bar{V}_2$. This isomorphism induces an isomorphism of $\F_p$-algebras
$$\End_{\F_p}((\F_p^{B})^{0})=\End_{\F_p}(\bar{V}_1)\otimes_{\F_p}\End_{\F_p}(\bar{V}_2),$$
under which (the images of) the subalgebras
$$R= \End_{\F_p}(\bar{V}_1)\otimes 1,  \ \tilde{R}=1\otimes  \End_{\F_p}(\bar{V}_2)$$
are $\mathrm{Alt}(B)$-normal subalgebras of $\End_{\F_p}((\F_p^{B})^{0})$, see \cite[Example 3.1(ii)]{ZarhinClifford}. In particular, the
$\mathrm{Alt}(B)$-module $(\F_p^{B})^{00}=(\F_p^{B})^{0}$ is {\sl not} very simple.

On the other hand, it follows from Theorem \ref{AdGH} that if $R$ is a {\sl non}-obvious $\mathrm{Alt}(B)$-normal subalgebra of $\End_{\F_p}((\F_p^{B})^{0})$
then $R \cong \Mat_a(\F_p)$ where a positive integer $a$ is a {\sl proper} divisor of 
$$\dim_{\F_p}((\F_p^{B})^{0})=4=2^2.$$
This implies that $a=2$ and $R \cong \Mat_2(\F_p)$.

The  assertion (iv) of Theorem \ref{A5} follows readily from already proven (ii) and (iii).
\end{proof}

\begin{thm}
\label{mathieuCentral}
Let $n \in \{11,12,22,23,24\}$.  Let $B$ be an   $n$-element set $B$ and 
$G\subset \mathrm{Perm}(B)$ the corresponding Mathieu group
  $\mathbf{M}_n$, which acts doubly transitively on $B$.
Let $p$ be an odd prime. If $n=11$, then we assume additionally that $p>3$.

Then 
\begin{itemize}
\item
the $G$-module $(\F_p^{B})^{00}$ is central simple;
\item
if $n \ne p+1$ and $n-1$ is divisible by $p$, then the $G$-module $(\F_p^{B})^{00}$ is very simple. 
\end{itemize}
\end{thm}

\begin{proof}
It follows from (\cite{Klemm}, \cite[Table 1]{Mortimer})  that
the  absolutely simple $G=\mathbf{M}_n$-module $(\F_p^{B})^{00}$ is absolutely simple.
By \cite{Atlas}  the index of every maximal subgroup of $\mathbf{M}_n$ is at least 
$$n>n-1  \ge N=\dim_{\F_p}\left((\F_p^{B})^{00})\right)$$
(recall that $N$ is either $n-1$ or $n-2$).

In light of Proposition \ref{centralsimpleIndex}, the $G=\mathbf{M}_n$-module $(\F_p^{B})^{00}$ is central simple.
It remains to prove the very simplicity in
 the ``exceptional'' cases when $n \ne p+1$ and $n-1$ is divisible by $p$.
We prove that in all the exceptional  cases the $\mathbf{M}_n$-module $(\F_p^{B})^{00}$ is {\sl very} simple. After that the desired
result will follow from Theorem  \ref{handysup}.
\begin{itemize}
\item
$n=11$. Then $p=5$ and $n-1=2 \times 5$ where both $2$ and $5$ are primes. If the $\mathbf{M}_{11}$-module $(\F_5^{B})^{00}$ 
is {\sl not} very simple then it follows from Theorem \ref{AdGH}(iii-d)  that there is a group embedding $\mathbf{M}_{11}\hookrightarrow \mathrm{PSL}(2,\F_5)$, which is not true.
Hence, the $\mathbf{M}_{11}$-module $(\F_5^{B})^{00}$ 
is  very simple. 
\item
$n=12$. Then $n-1$ is a prime and there are no exceptional cases.
\item
$n=22$. Then $n-1=22-1=3\cdot 7$ where both $3$ and $7$ are primes.  Then $p=3$ or $7$. If the $\mathbf{M}_{22}$-module $(\F_{3}^{B})^{00}$  
is {\sl not} very simple then  it follows from Theorem \ref{AdGH}(iii-d)  that 
there is a group embedding $\mathbf{M}_{22}\hookrightarrow \mathrm{PSL}(3,\F_p)$. Such an embedding does not exist
if $p=3$, because the order of $\mathrm{PSL}(3,\F_3)$ is {\sl not} divisible by $11$ while $11$ divides the order of $\mathbf{M}_{22}$.
Hence, the $\mathbf{M}_{22}$-module $(\F_3^{B})^{00}$ 
is  very simple.

 Such an embedding does not exist
if $p=7$ as well, because the order of $\mathrm{PSL}(3,\F_7)$ is {\sl not} divisible by $11$, which divides the order of $\mathbf{M}_{22}$.
Hence, the $\mathbf{M}_{22}$-module $(\F_7^{B})^{00}$  
is also  very simple.
\item
$n=23$. Then $n-1=22=2\cdot 11$ where both $2$ and $11$ are primes. Then $p=11$. If the $\mathbf{M}_{23}$-module $(\F_{11}^{B})^{00}$  
is {\sl not} very simple then  it follows from Theorem \ref{AdGH}(iii-d)  that 
 there is a group embedding $\mathbf{M}_{23}\hookrightarrow \mathrm{PSL}(2,\F_{11})$. Such an embedding does not exist.
Hence, the $\mathbf{M}_{23}$-module  $(\F_{11}^{B})^{00}$  
is  very simple.
\item
$n=24$. Then $n-1=23$ is a prime and there are no exceptional cases.
\end{itemize}
\end{proof}

\begin{prop}
\label{HSCo3}
Let $G$ be a doubly transitive permutation subgroup of a $n$-element set $B$.  Let $p>3$ be a prime.
Suppose that $(n,G)$ enjoys one of the following properties.

\begin{enumerate}
\item[(1)]
 $n=176$ and $G$ is isomorphic to $\mathrm{HS}$;
\item[(2)]
 $n=276$ and $G$ is isomorphic to $\mathrm{Co}_3$.
\end{enumerate}
Then the $G$-module $(\F_p^{B})^{00}$ is central simple.
\end{prop}

\begin{proof}
It follows from \cite[Tables]{Mortimer} that in both cases the $G$-module $(\F_p^{B})^{00}$ is absolutely simple. 

{\bf Case 1}. According to the Atlas \cite{Atlas}, if $H$ is a maximal subgroup  of $\mathrm{HS}$ with index
$[\mathrm{HS}:H]<176$ then $[\mathrm{HS}:H]=100$, which divides neither $176-1$ nor $176-2$.
By Proposition  \ref{centralsimpleIndex}, the $G$-module $(\F_p^{B})^{00}$ is central simple.

{\bf Case 2}. According to the Atlas \cite{Atlas},  if $H$ is a maximal subgroup subgroup of $\mathrm{Co}_3$ then its index
$m=[\mathrm{Co}_3:H]$ is greater or equal than 276 \cite{Atlas}; in particular, it divides neither $276-1$ nor $276-2$.
By Proposition  \ref{centralsimpleIndex},  the $G$-module $(\F_p^{B})^{00}$ is  central simple as well.
\end{proof}

\begin{thm}
 \label{PSL2Central}
 \begin{itemize}
 \item[(i)]
 Let $\ell$ be a prime, $\mathfrak{r}$ a positive integer, and 
  $n=\mathfrak{q}+1$ where $\mathfrak{q}=\ell^{\mathfrak{r}}>11$.

\item[(ii)]
 Let $G$ be a subgroup of $\mathrm{Perm}(B)$.
 Suppose that $G$ contains a subgroup $H$ that
 is isomorphic to $\mathbf{L}_2(\mathfrak{q})=\mathrm{PSL}(2,\F_{\mathfrak{q}})$ where $\F_{\mathfrak{q}}$ is a $\mathfrak{q}$-element field.
 \end{itemize}
 
  If $p$ is an odd prime
 then the $G$-module  $(\F_p^{B})^{00}$ is central simple.
 \end{thm}
 
 \begin{proof}
 It suffices to check that the $H\cong \mathbf{L}_2(q)$-module  $(\F_p^{B})^{00}$ is central simple.
First, our conditions on $q$ imply that each subgroup of $\mathbf{L}_2(\mathfrak{q})$ (except $\mathbf{L}_2(\mathfrak{q})$ itself) has index $\ge \mathfrak{q}+1=n$
\cite[p. 414, (6.27)]{Suzuki}. This implies that $H$ acts transitively on the $(\mathfrak{q}+1)$-element
set $B$ and the stabilizer $H_b$ of any $b \in B$ has index $\mathfrak{q}+1$. It follows  from \cite[Th. 6.25 on p. 412]{Suzuki}
that $H_b\subset \mathbf{L}_2(\mathfrak{q})$ is  conjugate to the (Borel) subgroup of upper-triangular matrices modulo $\{\pm 1\}$.
It follows that the $ \mathbf{L}_2(\mathfrak{q})$-set $B$ is isomorphic to the {\sl projective line} $\mathbb{P}^1(\F_{\mathfrak{q}})$ with the standard
fractional-linear action of $ \mathbf{L}_2(\mathfrak{q})$, which is doubly transitive.

Notice that 
$$q+1=n>N=\dim_{\F_p}\left((\F_p^{B})^{00})\right),$$
because $N$ is either $n-1$ or $n-2$. It follows that the index of any maximal subgroup of $\mathbf{L}_2(\mathfrak{q})$ does {\sl not} divide $N$.
On the other hand, according to \cite[Table 1]{Mortimer}, the $H=\mathbf{L}_2(\mathfrak{q})$-module is absolutely simple. It follows now from Proposition
\ref{centralsimpleIndex} that the $H$-module $V$ is central simple.
 \end{proof}

\begin{thm}
\label{deformation}
Let $O$ be a Dedekind ring,  $T$ a locally free/projective $O$-module of finite positive rank $r$. Let $E$ be the field of fractions
of $O$, $\mathfrak{m}$ a maximal ideal in $O$ and $k=O/\mathfrak{m}$ its residue field. Let us consider the $r$-dimensional $E$-vector space
$T_E=T\otimes_O E$ and  the $r$-dimensional $k$-vector space
$T_{k}=T/\mathfrak{m}T=T\otimes_O k$.

Let $G$ be a group and
$\rho:G \to \Aut_{O}(T)$ be a group homomorphism ($O$-linear representation). Let us consider the corresponding $E$-linear representation of $G$
$$\rho_E: G \to \Aut_E(T_E),  \  \sigma \mapsto \{ t\otimes e \mapsto \rho(\sigma)(t)\otimes e \ \forall t\in T,e\in E\} \ \forall \sigma\in G$$
and the corresponding $k$-linear representation of $G$
$$\rho_{k}: G \to \Aut_{k}(T_k),  \  \sigma \mapsto \{ t\otimes c \mapsto \rho(\sigma)(t)\otimes c \ \forall t\in T, c \in k\} \ \forall \sigma\in G.$$

If $\rho_{k}$ is central simple (resp. very simple) then $\rho_E$ is  central simple (resp. very simple).

\end{thm}

\begin{proof}
We view $T=T\otimes 1$ as a certain  $G$-invariant lattice in $T\otimes_O E=T_E$ and 
$\End_O(T)=\End_O(T)\otimes 1$ as a certain  $G$-invariant lattice (subalgebra) in
$\End_O(T)\otimes_O E=\End_E(T_E)$.

Let $R_E$ be a $G$-normal  $E$-subalgebra of $\End_E(T_E)$. Let $C_E$ be the center of $R_E$ 
and $J_E$ a {\sl proper} ideal of $R_E$. We have
$$E\subset C, \ J \subset R_E, \ J \ne R_E.$$
Let us consider the $O$-subalgebra
$R:=R_E\cap \End_O(T)$.  Clearly,  $R$ is a saturated $O$-submodule of  $\End_O(T)$, i.e.,
the quotient  $\End_O(T)/R$ is torsion free (finitely generated) $O$-module.

The natural map
$$R\otimes_O E\to R_E, \ u \otimes e \mapsto e\cdot u$$
is an isomorphism of $E$-algebras. This implies that:

\begin{itemize}
\item[(i)]
$C:=C\cap  \End_O(T)$ is the center of $R$   that contains $O$ as a saturated $O$-submodule, i.e.,  the quotient $C/O$ is a 
 torsion free (finitely generated) $O$-module.  In addition, $C$ is a saturated $O$-submodule of $R$, i.e., $R/C$ is a torsion free
 finitely generated $O$-module.
\item[(ii)]
The intersection  $J:=J_E\cap  \End_O(T)$ is
a {\sl proper} ideal of $R$  that is a saturated $O$-submodule of $R$,
i.e., the quotient $R/J$ is a  torsion free (finitely generated) $O$-module.
\end{itemize}
Since  the $O$-modules $\End_O(T)/R$, $C/O$, $R/C$ and $R/J$ finitely generated torsion free, they are {\sl projective}, because the ring $O$ is  {\sl Dedekind}. This implies that there are locally free submodules
$R_1\subset \End_O(T)$, $C_1\subset C$,  $D\subset R$, and $I\subset R$ such that
\begin{equation}
\label{splittingVery}
\End_O(T)=R\oplus R_1, \ C=O\oplus C_1, \  R=C\oplus D, \  I\oplus J=R.
\end{equation}
Since $J$ is a proper ideal, $I \ne 0$.
Since $C_1$ and $J$ are torsion-free finitely generated $O$-modules they are also locally free/projective.
Now let us consider the $k=O/\mathfrak{m}$-subalgebra
$$R_{k}= R\otimes_O k \subset \End_O(T)\otimes_O k=\End_{k}(T_{k})$$
where $T_k:=T\otimes_O k$.
Clearly,
\begin{enumerate}
\item[(1)]
 $R_{k}$ is a $G$-normal subalgebra of $\End_{k}(T_{k})$;
 \item[(2)]
 $k\oplus (C_1\otimes_O k)$ lies in the center of $R_{k}$.
 \item[(3)]
 $R_{k}=(J\otimes_O k)\oplus (I\otimes_O k)\ne \{0\}$.
 This implies that
 $J_{k}=J\otimes_O k$ is a {\sl proper} two-sided ideal of $R_{k}$,
 because $I \ne \{0\}$ and therefore $I\otimes_O k\ne \{0\}$.
 \end{enumerate}
 Suppose that $\rho_{k}$ is  central simple. Then $R_{k}$ is a simple 
 $k$-algebra with center $k$. It follows that
 $$C_1\otimes_ O k=\{0\}, \ J_{k}=J\otimes_ O k=\{0\}.$$
 Since $C_1$ and $J$ are locally free, we conclude that
 $$C_1=\{0\},  \ J=\{0\},$$ 
 which implies that 
 $$J_E=\{0\},  C=O\oplus C_1=O\oplus \{0\}=O$$
 and therefore $C_E=E$. This means that $T_E$ is a central simple $E$-algebra, which proves 
 that $\rho_E$ is also central simple.
 
 Assume now that $\rho_{k}$ is  very simple. Then either $R_{k}=k$
 or $R_{k}=\End_{k}(T_{k})$.
 In the latter case, applying
 \eqref{splittingVery}, we get
 $$\End_{k}(T_{k})=(R\otimes_O k)\oplus (R_1\otimes_O k)=R_{k}\oplus (R_1\otimes_O k)=\End_{k}(T_{k})\oplus (R_1\otimes_O k).$$
 Now  $k$-dimension arguments imply that $R_1\otimes_O k=\{0\}$ and therefore $R_1=\{0\}$. This implies that
 $\End_O(T)=R\oplus R_1=R$ and therefore $R_E=\End_E(T_E)$.
 
 Assume now that $R_{k}=k$. It follows from \eqref{splittingVery} that
 $$R=O\oplus (C_1\oplus D)$$ and therefore
 $$k=R\otimes_O k=k\oplus  \ (C_1\oplus D)\otimes_O k.$$
 Again, $k$-dimension arguments imply that  $(C_1\oplus D)\otimes_O k=\{0\}$ and therefore
 $C_1\oplus D=\{0\}$. It follows that $R=O$ and therefore $R_E=E$. This proves that $\rho_E$ is semisimple.
 \end{proof}

\section{Abelian varieties and cyclotomic fields}
\label{cyclotomicA}

Let $p$ be  a prime, $r$ a positive integer, and $q=p^r$. Let $E=\Q(\zeta_q)$ be the $q$th  cyclotomic field
and $O_E=\Z[\zeta_q]$ its ring of integers.

Let us put 
$$\eta=\eta_q:=1-\zeta_q\in \Z[\zeta_q].$$
It is well known \cite{W} that the principal ideal $\eta_q \Z[\zeta_q]$ of $\Z[\zeta_q]$ is maximal and contains $p \Z[\zeta_q]$. Actually,
$$p \Z[\zeta_q] =\eta_q^{\phi(q)}  \Z[\zeta_q].$$
It follows that there is $\eta^{\prime} \in  \Z[\zeta_q]$ such that
\begin{equation}
\label{cycloPrime}
\eta^{\prime}  \Z[\zeta_q]=\eta_q^{\phi(q)-1}  \Z[\zeta_q], \ \eta_q \eta^{\prime}=\eta^{\prime}\eta_q=p.
\end{equation}
The residue field 
$\Z[\zeta_q]/\eta \Z[\zeta_q]$ coincides with $\F_p$.
It is also well known \cite{W} that 
$$\Z_p[\zeta_q]=\Z[\zeta_q]\otimes \Z_p$$
 is the ring of integers in the $p$-adic   $q$th cyclotomic field $\Q_p(\zeta_q)$
and $\eta_q \Z_p[\zeta_q]$ is the maximal ideal of $\Z_p[\zeta_q]$ with residue field 
$$\Z_p[\zeta_q]/\eta_q \Z_p[\zeta_q]=\Z[\zeta_q]/\eta_q \Z[\zeta_q]=\F_p.$$

Let $K$ be a field of characteristic different from $p$. Let $K_a$ be the algebraic closure of $K$ and $K_s\subset K_a$ the separable
algebraic closure of $K$. We write $\Gal(K)$ for the automorphism group $\Aut(K_a/K)=\Gal(K_s/K)$ of the corresponding field extension.

Let $Z$ be an  abelian variety of positive dimension $g$ over $K$, and $\End_K(Z)$ (resp. $\End(Z)$) the ring of its $K$-endomorphisms
(resp. the ring of all $K_a$-endomorphisms). By a theorem of Chow, all endomorphisms of $Z$ are defined over $K_s$.
In addition, $Z[p]\subset Z(K_s)$ where $Z[p]$ is the kernel of multiplication by $p$ in $Z(K_a)$. If $m$ is an integer then we write
$m_Z\in \End_K(Z)$ for multiplication by $m$ in $Z$.

Suppose that we are given the ring embedding
$${\bf i}:O_E\hookrightarrow \End_K(Z)\subset \End(Z)$$
such that $1\in O_E$ goes to the identity automorphism $1_Z$ of $Z$.
In light of \eqref{cycloPrime},  ${\bf i}(\eta_q): Z \to Z$  and  ${\bf i}(\eta^{\prime}): Z \to Z$ are isogenies and the kernel
$\ker({\bf i}(\eta_q))$  of ${\bf i}(\eta_q)$ lies in $Z[p]$. In addition,
\begin{equation}
\label{kerEta}
\ker({\bf i}(\eta_q))={\bf i}(\eta^{\prime})(Z[p])\subset Z[p].
\end{equation}
Indeed, since $\eta_q \eta^{\prime}=p$,  we have  ${\bf i}(\eta_q){\bf i}( \eta^{\prime})=p_Z$ and
$${\bf i}(\eta^{\prime})(Z[p])\subset \ker({\bf i}(\eta_q)).$$
Conversely, suppose that $z \in \ker({\bf i}(\eta_q))$. Since ${\bf i}(\eta^{\prime})$ is an isogeny, it is surjective and therefore
there is $\tilde{z} \in Z(K_a)$ such that ${\bf i}(\eta^{\prime})(\tilde{z})=z.$ This implies that
$$0={\bf i}(\eta_q)(z)={\bf i}(\eta_q){\bf i}(\eta^{\prime})\tilde{z}=p_Z \tilde{z}=p \tilde{z}.$$
It follows  that $\tilde{z} \in Z[p]$ and therefore $\ker({\bf i}(\eta_q))\subset {\bf i}(\eta^{\prime})(Z[p])$,
which ends the proof of \eqref{kerEta}.

Let us put
\begin{equation}
\label{deltaQ}
\delt:={\bf i}(\zeta_q)\in \End_K(Z)\subset \End(Z).
\end{equation}

\begin{rem}
\label{invDeltaQ}
\begin{itemize}
\item[(i)]
Since $\eta_q=1-\zeta_q$, we get ${\bf i}(\eta_q))=1_Z-\delt$ and therefore
\begin{equation}
\label{invariantKer}
\ker({\bf i}(\eta_q))=\{z \in Z(K_a)\mid \delt(z)=z\}=:Z^{\delt}.
\end{equation}
\item[(ii)]
Since $\ker({\bf i}(\eta_q))$ is a subgroup of $Z[p]$, it carries the natural structure of a $\F_p$-vector space.
In other words,  $\ker({\bf i}(\eta_q))$  is a $\F_p$-vector subspace of $Z[p]$.
\item[(iii)]
Since the endomorphism  ${\bf i}(\eta_q)$ is defined over $K$, $\ker({\bf i}(\eta_q))$ is a $\Gal(K)$-invariant subspace  of $Z[p]$. The action of $\Gal(K)$
on $\ker({\bf i}(\eta))$  gives rise to the natural linear representation
\begin{equation}
\label{rhoEta}
\rho_{\eta}=\rho_{\eta,Z}: \Gal(K) \to \Aut_{\F_p}(\ker({\bf i}(\eta_q))),
\end{equation}
 $$ \sigma \mapsto \{z \mapsto \sigma(z) \ \forall z \in \ker({\bf i}(\eta_q))\subset Z[p]\subset Z(K_s)\} \ \forall
\sigma \in \Gal(K).$$
\end{itemize}
\end{rem}

\begin{lem}
\label{RibetEta}
$\phi(q)=[E:\Q]$ divides $2\dim(Z)=2g$ and $\ker({\bf i}(\eta_q))$ is a $\F_p$-vector space of dimension 
$$h_E:=\frac{2\dim(Z)}{[E:\Q]}=\frac{2g}{\phi(q)}.$$
\end{lem}

\begin{proof}
 By a result of Ribet \cite[Prop. 2.2.1 on p. 769]{Ribet2}, the
$\Z_p$-{\sl Tate module} $T_p(Z)$ of $Z$ is a free module over the ring
$$\Z_p[\delt]={\bf i}(O_E)\otimes \Z_p \cong \Z[\zeta_q]\otimes \Z_p =\Z_p[\zeta_q]$$
 of rank $h_E=2g/\phi(q)$. (In particular, $h_E$ is an integer.)  
 This implies that the $\Z_p[\delta]$-module 
$Z[p]=T_p(Z)/p T_p(Z)$  is isomorphic to $(\Z_p[\delt]/p)^{h_E}$.  It follows from \eqref{kerEta} that the $\F_p$-vector space
$$\ker({\bf i}(\eta_q))\cong \left(\eta^{\prime}\Z_p[\zeta_q]/p\right)^{h_E}=$$
$$\left(\eta^{\prime} \Z_p[\zeta_q]/ \eta^{\prime}\eta_q \Z_p[\zeta_q]\right)^{h_E}=
 \left(\Z_p[\zeta_q]/\eta_q\Z_p[\zeta_q]\right)^{h_E}=\F_p^{h_E}.$$
 This proves that $\ker({\bf i}(\eta_q))$ is a $\F_p$-vector space of dimension $h_E$.
\end{proof}

 Let $\Lambda$ be the centralizer of ${\bf i}(O_E)$ in $\End(Z)$.  Clearly, ${\bf i}(O_E)$ lies in the center of $\Lambda$.
 It is also clear that
 $$\Lambda (\ker {\bf i}(\eta_q)) \subset \ker({\bf i}(\eta_q)),$$
 which gives rise to the natural  homomorphism of $O/\eta_q O=\F_p$-algebras
 \begin{equation}
 \label{inMODeta}
 {\bf \kappa}: \Lambda/{\bf i}(\eta_q)\Lambda \to \End_{\F_p}(\ker {\bf i}(\eta_q)),  \ u+{\bf i}(\eta_q)\Lambda \mapsto \{z \mapsto u(z)\} \ \forall z \in \ker {\bf i}(\eta_q).
 \end{equation}
 
 \begin{prop}
 \label{inEta}
 The homomorphism ${\bf \kappa}$ defined in \eqref{inMODeta} is injective.
 \end{prop}
 
\begin{proof}
Suppose that $u \in \Lambda$ and $u(\ker {\bf i}(\eta_q))=\{0\}$. We need to prove that $u \in {\bf i}(\eta_q)\Lambda$.
In order to do that, notice that the endomorphism of $Z$
$$v:={\bf i}(\eta^{\prime})\ u=u \ {\bf i}(\eta^{\prime})\in \Lambda \subset \End(Z)$$
 kills $Z[p]$, because
$$v(Z[p])=u\ {\bf i}\left(\eta^{\prime})(Z[p]\right)= u \ ({\bf i}\left(\eta^{\prime})Z[p]\right)=u\big(\ker {\bf i}(\eta_q)\big)=\{0\}.$$
This implies that there is $\tilde{v} \in \End(Z)$ such that $v=p \tilde{v}$. Since $v$ commutes with ${\bf i}(O_E)$,
$\tilde{v}$ also commutes with ${\bf i}(O_E)$, i.e., $\tilde{v}\in \Lambda$. We have
$${\bf i}(\eta^{\prime}){\bf i}(\eta_q) \tilde{v}=p \tilde{v}=v={\bf i}(\eta^{\prime})u.$$
This implies that in $\End(Z)$
$${\bf i}(\eta^{\prime})\big({\bf i}(\eta_q) \tilde{v}-u\big)=0.$$
Multiplying it by ${\bf i}(\eta_q)$ from the left and taking into account that ${\bf i}(\eta_q){\bf i}(\eta^{\prime})={\bf i}(p)$, we get
$$p\big({\bf i}(\eta_q) \tilde{v}-u\big)=0$$
in $\End(Z)$.  It follows that 
 ${\bf i}(\eta_q) \tilde{v}=u$.
Since $\tilde{v}\in \Lambda$, we are done.
\end{proof}

\begin{rem}
\label{normalZ}
\begin{itemize}
\item[(i)]
Since $Z$ is defined over $K$, one may associate with every $u\in \End(Z)$ and $ \sigma \in \Gal(K)$ an endomorphism $^{\sigma}u \in \End(Z)$ such that 
\begin{equation}
\label{uSigma}
^{\sigma}u(z)=\sigma
u(\sigma^{-1}z) \quad \forall z \in Z(K_a).
\end{equation}
\item[(ii)]
Recall  that ${\bf i}(O_E)\subset \End_K(Z)$ consists of $K$-endomorphisms of $Z$. It follows that if $u \in \End(Z)$ commutes with  ${\bf i}(O_E)$ then 
$^{\sigma}u$ commutes with ${\bf i}(O_E)$ for all $\sigma \in \Gal(K)$. In other words, if $u \in \Lambda$ then $^{\sigma}u\in \Lambda$ for all $\sigma \in \Gal(K)$.
\item[(iii)] Since $O_E=\Z[\zeta_q]$, we have ${\bf i}(O_E)=\Z[\delt]$. 
It follows that $\Lambda$ coincides with the centralizer of $\delt$ in $\End(Z)$.
\end{itemize}
\end{rem}

\begin{prop}
\label{normalLambdaEta}
The image $R:={\bf \kappa}(\Lambda/\eta_q\Lambda)$ is a $\Gal(K)$-normal subalgebra of $\End_{\F_p}(\ker {\bf i}(\eta_q))$.
\end{prop}

\begin{proof}
\label{RnormalLambda}
Let $u \in \Lambda$. Then
$$\bar{u}: ={\bf \kappa}(u+\eta_q \Lambda)\in R \subset\End_{\F_p}(\ker {\bf i}(\eta_q)),$$
$$  \bar{u}: z \mapsto u(z) \ \forall z \in \ker {\bf i}(\eta_q).$$
Then $^{\sigma}u\in \Lambda$ for all $\sigma \in \Gal(K)$ and
$$\overline{^{\sigma}u}={\bf \kappa}(^{\sigma}u+\eta_q \Lambda)\in \End_{\F_p}(\ker {\bf i}(\eta_q)),$$
$$\overline{^{\sigma}u}: z \mapsto \sigma u \sigma^{-1} (z)=\rho_{\eta}(\sigma) u \rho_{\eta}(\sigma)^{-1}(z)=\rho_{\eta}(\sigma) \bar{u} \rho_{\eta}(\sigma)^{-1}(z).$$
In other words, for each $\bar{u}\in R$
$$\rho_{\eta}(\sigma) \bar{u} \rho_{\eta}(\sigma)^{-1} \in R \ \forall \sigma \in \Gal(K).$$
This proves that $R$ is $\Gal(K)$-normal.
\end{proof}

\begin{rem}
\label{LambdaQ}
\begin{itemize}
\item[(i)]
Extending ${\bf i}$ by $\Q$-linearity, we get a $\Q$-algebra embedding
$$E=O_E\otimes\Q \to \End(Z)\otimes \Q=:\End^0(Z),  \ u\otimes c \mapsto cu \ \forall u \in O, c \in \Q$$
that we continue to denote by ${\bf i}$. Clearly, ${\bf i}(E)$ coincides with the $\Q$-subalgebra $\Q[\delt]$ of $\End^0(Z)$ generated by $\delta_q$.
Clearly, ${\bf i}: E \to \Q[\delt]$  is a field isomorphism of number fields, and ${\bf i}(O_E)$ is the ring of integers in 
the number  field $\Q[\delta]$.
\item[(ii)]
Let us consider the $\Q$-subalgebra
$$\mathcal{H}=\Lambda \otimes \Q\subset \End(Z)\otimes \Q=\End^0(Z).$$
Then the center of $\mathcal{H}$ contains ${\bf i}(O)\otimes \Q=\Q[\delt]$. In other words, $\mathcal{H}$ is a $\Q[\delt]$-algebra of finite dimension. 
\item[(iii)]
We have
\begin{equation}
\label{HendLambda}
\Lambda=\mathcal{H}\cap \End(Z).
\end{equation}
where the intersection is taken in $\End^0(Z)$.
(Here we identify $\End(Z)$ with $\End(Z)\otimes 1$ in $\End^0(Z)$.)
Indeed, the inclusion $\Lambda\subset\mathcal{H}\cap \End(Z)$ is obvious.
Conversely, suppose that $u \in \mathcal{H}\cap \End(Z)$. Then $u \in \End(Z)$ and
$m u \in \Lambda$ for some positive integer $m$. This means that
$$(mu)\delt=\delt (mu),$$
which means that $m(u \delt-\delt u)=0$ in $\End(Z)$. It follows that $u \delt-\delt u=0$,
i.e., $u \in \Lambda$. It follows that $\mathcal{H}\cap \End(Z)\subset \Lambda$, which ends the proof of \eqref{HendLambda}.

\end{itemize}
\end{rem}

\begin{prop}
\label{normalLambdaDim}
\begin{itemize}
\item[(i)]
If the $\Gal(K)$-module $\ker {\bf i}(\eta_q)$ is central simple then 
$\mathcal{H}$ is a central simple $\Q[\delt]$-algebra.

\item[(ii)]
If the $\Gal(K)$-module $\ker {\bf i}(\eta_q)$ is very simple then either
$$\mathcal{H}=\Lambda\otimes\Q={\bf i}(E)=\Q[\delt], \ \Lambda={\bf i}(O_E)=\Z[\delt]$$ or 
$\mathcal{H}=\Lambda\otimes\Q$ is a central simple $\Q[\delt]$-algebra, whose dimension is  the square of $2\dim(Z)/[E:\Q]=2g/\phi(q)$.
\end{itemize}
\end{prop}

\begin{proof}
\begin{itemize}
\item[(i)]
The central simplicity implies that the $\Gal(K)$-normal subalgebra 
$$R={\bf \kappa}(\Lambda/\eta\Lambda)\cong \Lambda/\eta\Lambda$$
 is a central simple $\F_p$-algebra and therefore 
is isomorphic to the matrix algebra $\Mat_d(\F_p)$ of a certain size $d$.  Applying Lemma \ref{redMatrix}
to $O={\bf i}(O_E)$, the maximal ideal $\mathfrak{m}={\bf i}(\eta_q O_E)$ and the residue field $k=\F_p$, we conclude
that $\mathcal{H}$ is a central simple $\Q[\delta_q]$-algebra.
\item[(ii)]
The very simplicity implies that either $\Lambda/\eta_q\Lambda=\F_p$ or
$$\Lambda/\eta\Lambda \cong \End_{\F_p}(\ker {\bf i}(\eta_q)) \cong \Mat_{h_E}(\F_p).$$
In the latter case, Lemma \ref{redMatrix} tells us that $\mathcal{H}$ is a central simple $\Q[\delt]$-algebra of dimension $h_E^2$.

In the former case, Lemma \ref{redMatrix} tells us that $\mathcal{H}$ is a central simple $\Q[\delt]$-algebra of dimension $1$,
i.e.,  $\mathcal{H}=\Q[\delt]$. Hence,
$$\Z[\delt] \subset \Lambda \subset \Q[\delt].$$
Since $\Z[\delt]\cong \Z[\zeta_q]$ is integrally closed and $\Lambda$ is a free $\Z$-module of finite rank, $\Z[\delt] = \Lambda$.

\end{itemize}

\end{proof}

\section{Cyclic covers and Jacobians}
\label{Prelim}

 Hereafter we fix an odd  prime $p$.

 Let us  assume
that $K$ is a subfield of $\C$.
We write $K_a$ for
the algebraic closure of $K$ in $\C$ and write $\Gal(K)$ for the absolute
Galois group $\Aut(K_a/K)$. We also fix in $K_a$ a primitive $p$th
root of unity $\zeta=\zeta_p$.

Let $f(x) \in K[x]$ be a separable polynomial of degree $n \ge 4$.
We write $\R_f$ for the $n$-element set of its roots and denote by
$L=L_f=K(\R_f)\subset K_a$ the corresponding splitting field of $f(x)$. As
usual, the Galois group $\Gal(L/K)$ is called the Galois group of
$f$ and denoted by $\Gal(f)$. Clearly, $\Gal(f)$ permutes elements
of $\R_f$ and the natural map of $\Gal(f)$ into the group
$\Perm(\R_f)$ of all permutations of $\R_f$ is an embedding. We
will identify $\Gal(f)$ with its image and consider it as the certain
permutation group of $\R_f$. Clearly, $\Gal(f)$ is transitive if
and only if $f$ is irreducible in $K[x]$. Therefore the
$\Gal(f)$-module $(\F_p^{\R_f})^{00}$ is defined. The canonical
surjection $$\Gal(K) \twoheadrightarrow \Gal(f)$$ provides
$(\F_p^{\R_f})^{00}$ with the canonical structure of the
$\Gal(K)$-module via the composition $$\Gal(K)\twoheadrightarrow
\Gal(f)\subset \Perm(\R_f) \subset \Aut((\F_p^{\R_f})^{00}).$$
Let us put 
\begin{equation}
\label{Vfp}
V_{f,p}:=(\F_p^{\R_f})^{00}.
\end{equation}

 Let $C=C_{f,p}$ be the smooth projective model of the smooth
affine $K$-curve
            $$y^p=f(x).$$

The genus $$g=g(C)=g(C_{f,p})$$ of $C$   is
$(p-1)(n-1)/2$ if $p$ does {\sl not} divide
$p$ and $(p-1)(n-2)/2$ if it does  (\cite{Koo}, pp. 401--402, \cite{Towse}, Prop. 1 on p. 3359,
 \cite{Poonen}, p. 148).

 Assume that $K$ contains $\zeta$. There is a non-trivial biregular
 automorphism of $C$
 $$\delta_p:(x,y) \mapsto (x, \zeta y).$$
Clearly, $\delta_p^p$ is the identity selfmap of $C$.

Let 
$$J^{(f,p)}:=J(C)=J(C_{f,p})$$
 be the Jacobian of $C$. It is a $g$-dimensional abelian variety defined over $K$ and one may view $\delta_p$ as an element of
 $$\Aut(C) \subset\Aut(J(C)) \subset \End(J(C))$$
such that
  $$\delta_p \ne \Id, \quad \delta_p^p=\Id$$
where $\Id$ is the identity endomorphism of $J(C)$.
Here $\End(J(C))$ stands for the ring of all $K_a$-endomorphisms of $J(C)$. As usual, we write $\End^0(J(C))=\End^0(J^{(f,p)})$ for
the corresponding $\Q$-algebra $\End(J(C))\otimes \Q$.

Recall  \eqref{embedP} that there is a ring embedding
$${\bf i}_{p,f}: \Z[\zeta_p] \cong \Z[\delta_p] \subset \End(J^{(f,p)}), \ \zeta_p  \mapsto \delta_p.$$
 Let us put
\begin{equation}
\label{kerEtaJ}
J^{(f,p)}(\eta_p)=:\ker({\bf i}_{p,f}(\eta_p))\subset J^{(f,p)}(K_a)
\end{equation}
where $\eta_p=1-\zeta_p\in \Z[\delta_p] $ (Section \ref{cyclotomicA}).

\begin{rem}
\label{HdeltaLambda}
Let 
$$\Lambda:=\End_{\delta_p}(J^{(f,p)})$$ 
be the centralizer of $\delta_p$ in $\End(J^{(f,p)})$.  Clearly,
$$\mathcal{H}:=\Lambda\otimes\Q \subset \End(J^{(f,p)})\otimes \Q\subset \End^0(J^{(f,p)})$$
is the centralizer of $\Q[\delta_p]$ in $\End^0(J^{(f,p)})$.
\end{rem}

\begin{thm}[Prop. 6.2 in \cite{Poonen}, Prop. 3.2 in \cite{SPoonen}]
\label{kereta}
There is a canonical isomorphism of the $\Gal(K)$-modules
$$ J^{(f,p)}(\eta_p) \cong V_{f,p}.$$
\end{thm}

\begin{rem}
\label{Galf}
Clearly, the natural homomorphism $\Gal(K) \to
\Aut_{\F_p}(V_{f,p})$ coincides with the composition
 $$\Gal(K)\twoheadrightarrow
\Gal(f)\subset \Perm(\R_f) \subset \Aut\left((\F_p^{\R_f})^{00}\right)= \Aut_{\F_p}(V_{f,p}).$$
\end{rem}

\begin{cor}
\label{normalLambdaCor}
\begin{itemize}
\item[(i)]
If the $\Gal(f)$-module $V_{f,p}$ is central simple then the $\Gal(K)$-module  $J^{(f,p)}(\eta_p)$ is central simple and 
$\mathcal{H}=\Lambda\otimes\Q$ is a central simple $\Q[\delta_p]$-algebra.

\item[(ii)]
If the $\Gal(f)$-module $V_{f,p}$  is very simple then  the  $\Gal(K)$-module  $J^{(f,p)}(\eta_p)$ is very simple and either
$\Lambda={\bf i}(O)$ or 
$\mathcal{H}=\Lambda\otimes\Q$ is a central simple $\Q[\delta_p]$-algebra, whose dimension is  the square of $2\dim(J^{(f,p)})/(p-1)$.
\end{itemize}
\end{cor}

\begin{proof}
It follows from Remark \ref{image}(ii) combined with Theorem \ref{kereta} that if the $\Gal(f)$-module  $V_{f,p}$ is central simple (resp. very simple) then
the $\Gal(K)$-module  $J^{(f,p)}(\eta_p)$  (defined in \eqref{kerEtaJ}) is central simple (resp. very simple). Now the desired result follows readily from
Proposition \ref{normalLambdaDim} applied to $Z=J^{(f,p)}$, $q=p$,  and ${\bf i}={\bf i}_{p,f}$.
\end{proof}

The following assertion was proven in \cite[Th. 3.6]{ZarhinCambridge}.

\begin{thm}
\label{Cambridge36}
Suppose that $n \ge 4$.  Assume that $\Q[\delta_p]$ is a maximal commutative subalgebra of
 $\End^0(J^{(f,p)})$. 
 
 Then  $\End^0(J^{(f,p)})=\Q[\delta_p]\cong \Q(\zeta_p)$ and therefore
 $\End(J^{(f,p)})=\Z[\delta_p]\cong \Z[\zeta_p]$.
\end{thm}

\begin{thm}
\label{handysup} Let $p$ be an odd prime and $\zeta \in K$. 
Suppose that the $\Gal(f)$-module $(\F_p^{\R_f})^{00}$
enjoys one of the following properties.

\begin{itemize}
\item[(i)]
The $\Gal(f)$-module $V_{f,p}=(\F_p^{\R_f})^{00}$  is very simple.
\item[(ii)]
The  $\Gal(f)$-module 
$V_{f,p}=(\F_p^{\R_f})^{00}$ is central simple.  In addition, 
either $n=p+1$, or $n-1$ is not divisible by $p$. 
\end{itemize}

  Then $\End^0(J^{(f,p)})=\Q[\delta_p]$ and $\End(J^{(f,p)})=\Z[\delta_p]$.

\end{thm}

\begin{proof}[Proof of Theorem \ref{handysup}]
In light of Theorem \ref{Cambridge36},
it suffices to check that $\Q[\delta_p]$ coincides with its own centralizer in
$\End^0(J^{(f,p)})$.
Recall that $J^{(f,p)}$ is a $g$-dimensional
abelian variety defined over $K$.

The properties of the $\Gal(f)$-module $(\F_p^{\R_f})^{00}$ and the integers $n,p$  imply (thanks to Remark \ref{image}(ii)) that either
the $\Gal(K)$-module $J^{(f,p)}(\eta_p)$ is  very simple, or  the following conditions hold.
\begin{itemize}
\item[(a)]
The $\Gal(K)$-module $J^{(f,p)}(\eta_p)$ is central simple.
\item[(b)]
Either $n=p+1$, or $n-1$ is not divisible by $p$.
\end{itemize}
In all the cases the normal $\F_p$-subalgebra $R\cong \Lambda/\eta_p\Lambda$ is isomorphic
to the matrix algebra $\Mat_d(\F_p)$ for some positive integer $d$.

Applying Corollary \ref{normalLambdaDim},
we conclude that $\mathcal{H}=\Lambda_{\Q}=\Lambda\otimes\Q$ is a central simple $\Q[\delta_p]$-algebra of dimension $d^2$ for some positive integer $d$.
In addition, if  the $\Gal(K)$-module $J^{(f,p)}(\eta_p)$ is  very simple, then either 
$$d=1, \ \mathcal{H}=\Q[\delta_p], \Lambda=\Z[\delta_p]$$
or 
$$d=2g/(p-1).$$ 
According  to   Remark \ref{ssLambda}(vi), $ d \ne 2g/(p-1)$. So, in the very simple case
$\mathcal{H}=\Q[\delta_p], \Lambda=\Z[\delta_p]$.

Now suppose that $J^{(f,p)}(\eta_p)$ is  {\sl not} very simple. Then either $n=p+1$ or $n-1$ is not divisible by $p$. It follows from  Remark \ref{ssLambda}(iv) that $d=1$. This implies that 
$H=\Q[\delta_p]$.  Therefore 
$$\Z[\delta_p]\subset \Lambda \subset \Q[\delta_p].$$
This implies that $\Lambda=\Z[\delta_p]$ and therefore the centralizer of $\Q[\delta_p]$
 in $\End^0(J^{(f,p)})$ coincides with $\Lambda\otimes\Q=\Q[\delta_p]$.

\end{proof}

\begin{thm}
\label{SimpleDT}
Let $n \ge 5$ be an integer, $p$ an odd prime, and $K$ contains a primitive $p$th root of unity. 
 Let us put $N:=n-1$ if $p$ does not divide $n$ and $N:=n-2$ if $p\mid n$.
Suppose that the Galois group $\Gal(f)$ of $f(x)$ contains a subgroup $H$  such that the representation of $H$ in $(\F_p^{\R_f})^{00}$ is absolutely irreducible. 
Assume additionally that 
\begin{itemize}
\item[(i)]
the index of every maximal subgroup of  $H$  does not divide $N$.
\item[(ii)]
Either $n=p+1$,  or $n-1$ is not divisible by $p$.
\end{itemize}

Then $\End^0(J^{(f,p)})=\Q[\delta_p]$ and $\End(J^{(f,p)})=\Z[\delta_p]$.
\end{thm}

\begin{proof}
[Proof of Theorem \ref{SimpleDT}]
Enlarging $K$ if necessary, we may and will assume that $H=\Gal(f)$. It follows from Proposition \ref{centralsimpleIndex} that
the  absolutely simple $H$-module $(\F_p^{\R_f})^{00}$ is central simple.
Applying Theorem \ref{handysup}, we conclude that
$\End^0(J^{(f,p)})=\Q[\delta_p]$ and $\End(J^{(f,p)})=\Z[\delta_p]$.
\end{proof}

\begin{rem}
See \cite[Sect. 7.7]{DM} and \cite{Mortimer} for the list of doubly transitive permutation groups $H\subset \Perm(B)$  and primes $p$ such that the $H$-module
$(\F_p^{B})^{00}$ is (absolutely) simple. (See also \cite[Sect. 4]{PraegerS}, \cite[Main Theorem]{Curtis}
and \cite{KT}.)
\end{rem}

\section{Jacobians of cyclic covers of prime degree $p$}
\label{mainPproof}

\begin{proof}[Proof of Theorem \ref{primeToP}]
Enlarging $K$ if necessary, we may and will assume that 
$$H=\Gal(f)\subset \Perm(\R_f).$$  
 Since $p>n$,
the prime $p$ divides neither $n$ nor $n-1$.
 In particular,
$$(\F_p^{\R_f})^{00}=(\F_p^{\R_f})^{0}.$$ 
In light of Remark \ref{DoubleprimeToP} applied to $B=\R_f$ and $G=H$, the double transitivity of $H$ implies that the $H$-module $(\F_p^{\R_f})^{0}$ is absolutely simple.
Now the desired result follows readily from Theorem \ref{SimpleDT}.
\end{proof}

\begin{proof}[Proof of Theorem \ref{endo}]
Assume that $n \ge 5$ and $\Gal(f)=\Perm(\R_f)$ or $\Alt(\R_f)$.
Enlarging $K$ if necessary, we may assume that $\Gal(f)=\Alt(\R_f)$.
Taking into account that $\Alt(\R_f)$ is non-abelian simple while the field extension $K(\zeta)/K$ is
abelian, we conclude that the Galois group of $f$ over $K(\zeta)$
is also  $\Alt(\R_f)$. (In particular, $f(x)$
remains irreducible over $K(\zeta)$.) So, in the course of the
proof of Theorem \ref{endo}, we may assume that $\zeta
\in K$ and $\Gal(f)=\Alt(\R_f)$.

 It is well known that
the index of every maximal subgroup of $\mathrm{Alt}(\R_f)\cong \mathbf{A}_n$ is at least $n$; notice that
$$n>N=\dim_{\F_p}(V_{f,p})=\dim_{\F_p}\left((\F_p^{B})^{00}\right).$$ (Recall that $N=n-1$ or $n-2$.)
By  Theorem \ref{A5}(iv), the $\Gal(f)$-module $V_{f,p}=(\F_p^{B})^{00}$ is {\sl central simple}.
It is {\sl very simple} if either $n>5$ or $p \le 5$, thanks to  Theorem \ref{A5}(ii).
On the other hand, if $n=5$ and $p>5$ then $n-1$ is {\sl not} divisible by $p$.
Now the desired result follows  readily from Theorem \ref{handysup}.

 \end{proof}

\begin{proof}[Proof of Theorem \ref{mathieu}]
Since $\mathbf{M}_n$, $\mathrm{HS}$ and $\mathrm{Co}_3$  are simple nonabelian groups, replacing $K$ by $K(\zeta)$, we may and will assume that $\zeta \in K$.
Now the desired result follows readily from   Theorem \ref{SimpleDT} combined with Theorem \ref{mathieuCentral} and Proposition \ref{HSCo3}.

\end{proof}

\begin{proof}[Proof of Theorem \ref{PSL2}]

Enlarging $K$, we may assume that $\Gal(f) =H\cong \mathfrak{G}(\mathfrak{q})$.
Since $\mathfrak{G}(\mathfrak{q})$ is a simple nonabelian group, replacing $K$ by $K(\zeta)$, we may and will assume that $\zeta \in K$.
 In light of \cite[Table 1]{Mortimer}, our conditions on $H$ and $p$   imply that the $H$-module $V_{f,p}$ is {\sl absolutely simple}. 

{\bf Case L2}. It follows from Theorem \ref{PSL2Central} that the $\Gal(f)$-module $(\F_p^{\R_f})^{00}=V_{f,p}$ is central simple.

On the other hand, if $n-1$ is divisible by $p$ then $p=\ell$, because $n-1=(\mathfrak{q}+1)-1=\mathfrak{q}$ which  is a power of the prime number $\ell$.
Hence, our assumptions imply that $n=\mathfrak{q}+1=p+1$.  So, either  $n-1$ is {\sl not} divisible by $p$ or $n=p+1$.
Now we may apply Theorem \ref{handysup}, which gives us
$\End^0(J^{(f,p)})=\Q[\delta_p]$ and $\End(J^{(f,p)})=\Z[\delta_p]$.

{\bf Case Lmq}. It follows from a result of Guralnick and Tiep \cite[Th. 1.1]{GurTiep} that every nontrivial projective representation
of $\Gal(f)=H=\mathrm{L}_m(\mathfrak{q})$ in characteristic $p$ has dimension $\ge \dim_{\F_p}(V_{f,p})$. 
In light of  \cite[Cor. 5.4]{ZarhinMMJ}, the $\Gal(f)$-module $V_{f,p}$ is very simple.
Now the desired result follows readily from  Theorem \ref{SimpleDT}.

{\bf Case U3}.
  It follows readily from the Mitchell's list of maximal subgroups of $\mathrm{U}_3(\mathfrak{q})$ (\cite[p. 212-213]{Hofer}, \cite[Th. 6.5.3 and its proof, pp. 329-332]{Gorenstein}  that the index of every maximal subgroup of  $\mathrm{U}_3(\mathfrak{q})$ is greater or equal than
  $$\mathfrak{q}^3+1 =n>N$$
  where $N=\dim_{\F_p}(V_{f,p})$ is either $n-1=\mathfrak{q}^3$ or $n-2=\mathfrak{q}^3-1$.
  On the other hand,  $n-1=\mathfrak{q}^3$ is a power of the prime $\ell$ and therefore is {\sl not} divisible by the prime $p$, since $\ell \ne p$.
  Now the desired result follows readily from  Theorem \ref{SimpleDT}.
  
  {\bf Case Sz}. It follows from the classification of subgroups of $\mathrm{Sz}(\mathfrak{q})$ \cite[Remark 3.12(e) on p. 194]{HB}
  that every maximal subgroup of $\mathrm{Sz}(\mathfrak{q})$ has index $ \ge \mathfrak{q}^2+1=n$.
  Since $n-1=\mathfrak{q}^2$ is a power of $2$, the odd prime $p$  does {\sl not} divide $n-1$.
  Now the desired result follows readily from  Theorem \ref{SimpleDT}.

  {\bf Case Ree}. Our conditions on $p$ imply that $p \ne 3$. 
   Since $n-1=\mathfrak{q}^2$ is a power of $3$, the  prime $p$  does {\sl not} divide $n-1$.
  It follows from the classification of subgroups of $\mathrm{Ree}(\mathfrak{q})$ \cite[Th. C]{Kleidman}
  that every maximal subgroup of $\mathrm{Sz}(\mathfrak{q})$ has index $ \ge \mathfrak{q}^3+1=n$.
  (See also \cite[Remark 5.4]{Tigran}.) Now the desired result follows readily from  Theorem \ref{SimpleDT}.
   \end{proof}

\section{Jacobians of cyclic covers of degree $q$}
\label{mainQproof}

In this section we discuss the case when $q=p^r>2$ where $r$ is any positive integer, $K$ is a subfield of $\C$ and $f(x) \in K[x]$ a degree $n$ polynomial without repeated roots.
We assume that $n \ge 5$ and either $q\mid n$ or $p$ does {\sl not} divide $n$. Let  $J(C_{f,q})$ be the Jacobian of the curve $C_{f,q}$ and $\delta_q$  the automorphism
of $J(C_{f,q})$,
whicch are defined
in the beginning of Section \ref{Introd}.

\begin{rem}
\label{factorQ}
One may define a positive-dimensional abelian subvariety 
$$J^{(f,q)}:=\mathcal{P}_{q/p}(\delta_q)(J(C_{f,q}))$$
 of $J(C_{f,q})$ \cite[p. 355]{ZarhinM}
that  is defined over $K(\zeta_q)$ and enjoys the following properties \cite{ZarhinM} (see also \cite{ZarhinGanita}).
\begin{itemize}
\item[(i)]
If $q=p$ then  $J^{(f,p)}=J(C_{f,p})$ (as above).
\item[(ii)]
 $J^{(f,q)}$ is defined over $K(\zeta_q)$.
\item[(iii)]
$J^{(f,q)}$ is  a $\delta_q$-invariant abelian subvariety of $J(C_{f,q})$. In addition $\Phi_q(\delta_q)(J^{(f,q)})=0$ where
$$\Phi_q(t)=\sum_{i=0}^{p-1} t^{i p^{r-1}} \in \Z[t]$$
is the $q$th cyclotomic polynomial. This gives rise to the ring embedding
$${\bf j}_{q,f}: \Z[\zeta_q] \hookrightarrow \End(J^{(f,q)})$$
under which $\zeta_q$ goes to to the restriction of $\delta_q$ to $J^{(f,q)}$, which we  denote by $\delt_q \in  \End(J^{(f,q)})$.
Then the subring $\Z[\delt_q]$ of $\End(J^{(f,q)})$ is isomorphic to $\Z[\zeta_q]$ (via ${\bf j}_{q,f}$), and the $\Q$-subalgebra
$\Q[\delt_q]$ of $\End^0(J^{(f,q)})$ is isomorphic to $\Q(\zeta_q)$.
\item[(iv)]
If $p$ does {\sl not} divide $n$ then there is an isogeny of abelian varieties
$$J(C_{f,q}) \to J(C_{f,q/p})\times J^{(f,q)}$$ that is defined over $K(\zeta_q)$. (Notice that $q/p=p^{r-1}$, so
$ J(C_{f,q/p})= J(C_{f,p^{r-1}})$.)
  By induction, this gives us an isogeny of abelian varieties
$$J(C_{f,q}) \to J(C_{f,p}) \times \prod_{i=2}^r J^{(f,r^i)}=\prod_{i=1}^r J^{(f,r^i)}$$
that is also defined over $K(\zeta_q)$ \cite[Cor. 4.12]{ZarhinM}.
\item[(v)]
Suppose that $\zeta_q \in K$. Then the $\Gal(K)$-submodule $\ker(\one-\delt_q)$ of $J^{(f,q)}(K_a)$ is isomorphic to $V_{f,p}$. 
(See  \cite[Lemma 4.11]{ZarhinM}, \cite[Th. 9.1]{ZarhinGanita}.) 
In particular, $\ker(\one-\delt_q)$  is a $N$-dimensional vector space over $\F_p$ where 
\begin{itemize}
\item
$N=n-1$ if $p$ does {\sl not} divide $n$;
\item
$N=n-2$ if $p$ divides $n$ and $q$  divides $n$.
\end{itemize}
(Here $\one$ stands for the identity automorphism of $J^{(f,q)}$.)
\item[(vii)]
Let us consider the action of the  subfield $\Q[\delt_q]$ of  $\End^0(J^{(f,q)})$  on $\Omega^1(J^{(f,q)})$. 
Let $i<q$ be a positive integer that is {\sl not} divisible by $p$  and
$\sigma_i:  \Q[\delt_q] \hookrightarrow \C$ be the field embedding that sends $\delt_q$ to $\zeta_q^{-i}$. Clearly,
$$\Omega^1(J^{(f,q)})=\oplus_i \Omega^1(J^{(f,q)})_{\sigma_i}$$
where $\Omega^1(J^{(f,q)})_{\sigma_i}$ are the corresponding weight subspaces (see Section \ref{weightS}).
Let us consider the nonnegative integers
$$n_{\sigma_i}:=\dim_{\C} (\Omega^1(J^{(f,q)})_{\sigma_i}).$$
\begin{enumerate}
\item[(1)]
If $p$ does {\sl not} divide $n$ then 
$$n_{\sigma_i}=\left[\frac{ni}{q}\right]$$
\cite[Remark 4.13]{ZarhinM}. In addition, the number of $i$ with $n_{\sigma_i}\ne 0$ is strictly greater than
$$\frac{(p-1)p^{r-1}}{2}=\frac{\phi(q)}{2}=\frac{[\Q[\delt_q]:\Q]}{2}$$ 
\cite[p. 357-358]{ZarhinM}).
\item[(2)]
If $p$ is odd and $q$ divides $n$
then the  GCD of all $n_{\sigma_i}$'s is $1$
\cite[Lemma 8.1(D]{ZarhinGanita}.
\item[(3)]
If $p$ is an odd prime that does {\sl not} divide $n$, and either $n=q+1$ or $n-1$ is {\sl not} divisible by $q$, then n the  GCD of all $n_{\sigma_i}$'s is $1$
\cite[Lemma 8.1(D]{ZarhinGanita}.
\end{enumerate}
\item[(viii)] If  $p$ is odd and  $\Q[\delt_q]$ is a maximal commutative subalgebra of $\End^0(J^{(f,q)})$ then
 $$\End^0(J^{(f,q)})=\Q[\delt_q], \ \End(J^{(f,q)})=\Z[\delt_q]$$
(\cite[Th. 4.16]{ZarhinM}, \cite[Th. 8.3]{ZarhinGanita}).
\end{itemize}
\end{rem}

\begin{thm}
\label{SimpleDTnew}
Let $n \ge 5$ be an integer, $p$ an odd prime, and $K$ contains a primitive $q$th root of unity.  Suppose that 
either $p$ does not divide $n$ or $q$ divides $n$.

 Let us put $N:=n-1$ if $p$ does not divide $n$ and $N:=n-2$ if $q\mid n$.
Suppose that the Galois group $\Gal(f)$ of $f(x)$ contains a subgroup $H$  such that the representation of $H$ in $\left(\F_p^{\R_f}\right)^{00}=V_{f,p}$ is absolutely irreducible. 
Assume additionally that the index of every maximal subgroup of  $H$  does not divide $N$  and one of the following conditions holds.
\begin{itemize}
\item[(i)]
 The representation of $H$ in $\left(\F_p^{\R_f}\right)^{00}=V_{f,p}$ is very simple. 
\item[(ii)]
Either $p$ does not divide $n$ and  $n-1$ is not divisible by $q$, or $n=q+1$, or $q\mid n$.
\end{itemize}

Then $\End^0(J^{(f,q)})=\Q[\delt_q]$ and $\End(J^{(f,q)})=\Z[\delt_q]$.
In particular, $J^{(f,q)}$ is an absolutely simple abelian variety.
\end{thm}

\begin{proof}
Enlarging $K$ if necessary, we may and will assume that $H=\Gal(f)$. It follows from Proposition \ref{centralsimpleIndex} that
the  absolutely simple $H$-module $(\F_p^{\R_f})^{00}=V_{f,p}$ is central simple. 

Recall (Remark \ref{factorQ}(v))  that  the $\Gal(K)$-module $\ker(1-\delt_q)$  is isomorphic to $V_{f,p}$ and therefore
is also central simple. In addition, it is very simple if and only if
the $H$-module $V_{f,p}$ is very simple.

Let $\Lambda$ be the centralizer of $\Z[\delt_q]$ in $\End(J^{(f,q)})$ and
$$\mathcal{H}=\Lambda_{\Q}:=\Lambda\otimes\Q$$
the centralizer of $\Q[\delt_q]$ in $\End^{0}(J^{(f,q)})$.
Applying Proposition \ref{normalLambdaDim} to 
$$Z=J^{(f,q)}, \ E=\Q[\delta_q], \  {\bf i}={\bf j}_{q,f}, \ O_E=\Z[\delt_q], $$
we conclude that $\mathcal{H}=\Lambda_{\Q}=\Lambda\otimes\Q$ is a central simple $\Q[\delt_q]$-algebra of dimension $d^2$ for some positive integer $d$.

In addition, if  the $\Gal(K)$-module $\ker(1-\delt_q)$ is  very simple, then either 
$$d=1, \ \mathcal{H}=\Q[\delt_q], \Lambda=\Z[\delt_q]$$
or 
$$d=N=\frac{2g}{\phi(q)}=\dim_{\F_p} (V_{f,p})$$ 
where 
$$g=\dim(J^{(f,q)}), \ \phi(q)=[\Q(\zeta_q):\Q]=[\Q[\delt_q]:\Q].$$
In light of   Remark \ref{factorQ}(i-ii) combined with Proposition \ref{normalLambdaDim}(ii), $ d \ne 2g/[\Q[\delt_q]:\Q]$. So, in the very simple case
$\mathcal{H}=\Q[\delt_q], \Lambda=\Z[\delt_q]$.

Now suppose that $\ker(1-\delt_q)$ is   {\sl not very} simple. Then $V_{p,f}$ is not very simple.
 This implies that either $n=q+1$,  or $p$ does not divide $n$ and $n-1$ is not divisible by $q$ or $q\mid n$. It follows from  
 Proposition \ref{normalLambdaDim}(i)
combined with Remark \ref{factorQ}(vii)
that $d=1$. This implies that 
$\mathcal{H}=\Q[\delt_q]$.  Therefore 
$$\Z[\delt_q]\subset \Lambda \subset \Q[\delt_q].$$
This implies that $\Lambda=\Z[\delt_q]$ and therefore $\mathcal{H}=\Q[\delt_q]$ is a maximal commutative subalgebra of
$\End^0(J^{(f,q)})$
 Now the desired result follows from Remark \ref{ssLambda}(viii).
\end{proof}

\begin{proof}[Proof of Theorems \ref{endoQ} and \ref{primeToQ}]

Enlarging $K$ if necessary, we may assume that $K$ contains a primitive $q$th root of unity, and

\begin{itemize}
\item
$\Gal(f)=\Alt(\R_f)=:H$ in the case of Theorem \ref{endoQ};
\item
 $\Gal(f)=H$ in the case of Theorem \ref{primeToQ}.
\end{itemize}
It follows from Theorem \ref{SimpleDTnew} that  $\End^0(J^{(f,q)})=\Q[\delt_q]\cong \Q(\zeta_q)$
and $J^{(f,q)}$ is an absolutely simple abelian variety
for $q=p^r$ when $r$ is any positive integer. This implies that for distinct positive integers $i$  and $j$ there are no nonzero homomorphisms
between $J^{(f,p^i)}$ and $J^{(f,p^j)}$, because they are absolutely simple abelian varieties with non-isomorphic endomorphism algebras.
This implies that the endomorphism algebra $\End^0(Y)$ of the product $Y:=\prod_{i=1}^r J^{(f,r^i)}$ is 
$\prod_{i=1}^r \Q(\zeta_{p^i})$, whose $\Q$-dimension is $q-1$.

In light of  Remark \ref{factorQ}(iv), 
if $p$ does {\sl not} divide $n$ then $Y$ is isogenous to $J(C_{f,q})$.
It follows that $\End^0(J(C_{f,q}))$ also has $\Q$-dimension $(q-1)$. However, we know that the $\Q$-algebra $\End^0(J(C_{f,q}))$
contains the $\Q$-subalgebra $\Q[\delta_q]$ of $\Q$-dimension $(q-1)$, thanks to \eqref{Qdeltaq}. This implies that
$$\End^0(J(C_{f,q}))=\Q[\delta_q]\cong \prod_{i=1}^r \Q(\zeta_{p^i}).$$
This ends the proof of Theorem \ref{primeToQ}.

Let us finish the proof of Theorem \ref{endoQ}, following \cite[Remark 4.3 and Proof of Th. 5.2 on p. 360]{ZarhinM}. It remains to do the case  when $q \mid n$, say, $n=qm$ for some positive integer $m$.
Since the case $q=p$ was already covered by already proven Theorem \ref{endo}, we may assume that
$q\ge p^2 \ge 9$ and therefore $n \ge 9$. Recall that $\Gal(f)=\Alt(\R_f) \cong \mathbf{A}_n$.
Let $\alpha \in K_a$ be a root of $f(x)$. Let us consider the overfield $K_1=K(\alpha)$ of $K$. We have
$f(x)=(x-\alpha)f_1(x)\in K_1[x]$ where $f_1(x)$ is a degree $(n-1)$ irreducible polynomial over $K_1$
with Galois group $\mathbf{A}_{n-1}$.  Let us consider the polynomials
$$h(x)=f_1(x+\alpha), \ h_1(x)=x^{n-1}\in K_1[x]$$
of degree $n-1 \ge 9-1=8$. Notice that $n-1$ is not divisible by $p$ and the Galois group of $h_1(x)$ over $K_1$
is still $\mathbf{A}_{n-1}$.   The standard substitution  
$$x_1=\frac{1}{x-\alpha}, \ y_1=\frac{y}{(x-\alpha)^m}$$
establishes a birational isomorphisms between the curves $C_{f,q}$ and $C_{h_1,q}$ \cite[p. 3359]{Towse}.
This implies that the Jacobians $J(C_{f,q})$ and $J(C_{h_1,q})$ are isomorphic and therefore their endomorphism algebras
are also isomorphic. Applying to $J(C_{h_1,q})$ the already proven part of Theorem \ref{endoQ}, we conclude
that the $\Q$-algebra $\End^0(J(C_{h_1,q}))$ has $\Q$-dimension $(q-1)$. This implies that $\End^0(J(C_{f,q}))$ also has $\Q$-dimension $q-1$. However, we know that  $\End^0(J(C_{f,q}))$
contains the $\Q$-subalgebra $\Q[\delta_q]$ of $\Q$-dimension $(q-1)$ \eqref{Qdeltaq}. This implies that 
$$\End^0(J(C_{f,q}))=\Q[\delta_q]\cong \prod_{i=1}^r \Q(\zeta_{p^i}).$$
This ends the proof of Theorem \ref{endoQ}.

\end{proof}

\end{document}